\RequirePackage{fix-cm}
\documentclass[smallextended]{svjour3}       
\smartqed  
\usepackage{graphicx}
\usepackage{amsmath}
\usepackage{amsfonts}
\usepackage{amssymb}
\usepackage{mathtools}
\usepackage{empheq}
\usepackage{pgf}
\usepackage{subfigure}
\usepackage{scrextend}

\def \to{\rightarrow}

\def \norm{\|}

\def \T {\mathbb{T}}
\def \E {\mathbb{E}}
\def \R {\mathbb{R}}

\def \Z {\mathbb{Z}}

\def \price{P}
\def \statespace{\T^d}

\def\1B{{\bf  1}}

\newcommand\be{\begin{equation}}
\newcommand\ee{\end{equation}}

\newcommand{\myprod}[2]
{ \left\langle #1 , #2 \right\rangle}

\def \smint {{\textstyle \int }}
\def \densities {\mathcal{D}_1(\statespace)}

\usepackage{enumitem, hyperref}
\makeatletter
\def\namedlabel#1#2{\begingroup
    #2%
    \def\@currentlabel{#2}%
    \phantomsection\label{#1}\endgroup
}
\makeatother

\begin{document}

\title{Schauder Estimates for a Class of Potential Mean Field Games of Controls\thanks{The first author acknowledges support from the FiME Lab (Institut Europlace de Finance).  The first two authors acknowledges support from the PGMO project ``Optimal control of conservation equations", itself supported by iCODE(IDEX Paris-Saclay) and the Hadamard Mathematics LabEx.\\ Conflict of Interest: The authors declare that they have no conflict of interest.
}
}


\author{J. Fr\'ed\'eric Bonnans \and Saeed Hadikhanloo \and Laurent Pfeiffer}


\institute{J. Fr\'ed\'eric Bonnans \at
              Inria-Saclay and Ecole Polytechnique, France \\
              \email{frederic.bonnans@inria.fr}           
           \and
           Saeed Hadikhanloo \at
              Inria-Saclay and Ecole Polytechnique, France \\
              \email{saeed.hadikhanloo@inria.fr} 
              \and
              Laurent Pfeiffer (corresponding author) \at
              University of Graz, Austria \\
              \email{laurent.pfeiffer@uni-graz.at}
}

\date{Received: date / Accepted: date}

\maketitle

\begin{abstract}
An existence result for a class of mean field games of controls is provided. In the considered model, the cost functional to be minimized by each agent involves a price depending at a given time on the controls of all agents and a congestion term. The existence of a classical solution is demonstrated with the Leray-Schauder theorem; the proof relies in particular on a priori bounds for the solution, which are obtained with the help of a potential formulation of the problem.
\end{abstract}

\section{Introduction}

The goal of this work is to prove the existence and uniqueness of a classical solution to
the following system of partial differential equations:
\begin{equation}\label{MFGC}
\left\lbrace 
\begin{array}{l l l}
(i) \quad &-\partial_t u -\sigma \Delta u + H(x,t, \nabla u(x,t) +
            \phi(x,t)^\intercal \price (t) )  \quad & \\[5pt]
             & \qquad \qquad \qquad \qquad \qquad \qquad \qquad \qquad = f(x,t,m(t))\quad &(x,t) \in Q, \\[5pt]
(ii) \quad &\partial_t m - \sigma \Delta m + \mathrm{div} (v m)= 0 \quad &(x,t) \in Q, \\[5pt]
(iii) \quad &\price(t) = \Psi \left(t, \int_{ \statespace} \phi(x,t) v(x,t) m(x,t) \; \mathrm{d} x \right) \quad &t \in [0,T],\\[5pt]
(iv) \quad & v(x,t) = - H_p(x,t, \nabla u (x,t) + \phi(x,t)^\intercal \price(t) )  &(x,t) \in Q, \\[5pt]
(v) \quad &m(x,0)= m_0(x), \quad u(x,T) = g(x) \quad & x \in \statespace,\\[5pt]
\end{array} 
\tag{MFGC} \right.
\end{equation}
where $u=u(x,t) \in \R$, $m=m(x,t) \in \R$, $v=v(x,t) \in \R^d$, $P= P(t) \in \R^k$,
with $(x,t) \in Q:=\statespace \times [0,T]$. The parameters $T>0$,
$\sigma >0$ are given and
\begin{align*}
\begin{array}{ll}
H \colon (x,t,p) \in Q \times \R^d \to \R, \qquad
& \Psi \colon (t,z) \in [0,T] \times \R^k \to \R^k, \\[5pt]
\phi \colon (x,t) \in Q \to \R^{k \times d},
& f \colon (x,t,m) \in Q \times \mathcal{D}_1(\statespace)\to \R, \\[5pt]
m_0 \in \densities,
& g \colon x \in \statespace\to \R
\end{array}
\end{align*}
are given data. The set $\mathcal{D}_1(\statespace)$ is defined as
\begin{equation} \label{eq:densities}
\mathcal{D}_1(\statespace)= \Big\{ m \in L^\infty(\statespace) \,|\, m \geq 0,\, \int_{\statespace} m(x) \; \mathrm{d} x =1 \Big\}.
\end{equation}
We work with $\Z^d$-periodic data and we set the
state set as the $d$-dimensional torus $\T^d$, that is a quotient set
$\R^d / \Z^d$. The Hamiltonian $H$ is assumed to be such that
$H(x,t,p)= L^*(x,t,-p)$, for some mapping $L$,
where $L^*(x,t,p)$
denotes the Fenchel transform with respect to $p$:
$$
  H(x,t,p) := \sup_{v \in \R^d } -\langle p,v\rangle - L(x,t,v).
$$
The mapping $L$ is assumed to be convex in its third variable.

The function $u$, as a solution to the
Hamilton-Jacobi-Bellman (HJB) in equation $(i)$\eqref{MFGC}
is the value function corresponding to the stochastic optimal control problem:
\begin{align}
 u(x,t)= \inf_\alpha \, \E \Big[ & \int_{t}^{T}
{L(X_s ,s, \alpha_s)} + \langle \phi(X_s , s)^\intercal P (s),\alpha_s
\rangle \; \mathrm{d} s \notag \\
& \qquad + \int_t^T f(X_s,s,m(s)) \; \mathrm{d} s + g(X_T)  \Big], \label{cost}
\end{align}
subject to the stochastic dynamics $\mathrm{d} X_s = \alpha_s \, \mathrm{d} s + \sqrt{2\sigma} \, \mathrm{d} B_s , \; X_t = x \in \statespace$. The feedback law $v$ given by $(iv)$\eqref{MFGC} is then optimal for this stochastic optimal control problem. Equation $(ii)\eqref{MFGC}$ is the Fokker-Planck equation which describes the evolution of the distribution $m(t)$ of the agents, when the optimal feedback law is employed.
At last, $(iii)\eqref{MFGC}$ makes the quantity $\price(t)$ endogenous. 

An interpretation of the system \eqref{MFGC} is as follows. Consider a
stock trading market. A typical trader, with an initial level of stock
$X_0=x$, controls its level of stock $(X_t)_{ t \in [0,T]}$ through
the purchasing rate $\alpha_t$ with stochastic dynamic $\mathrm{d} X_t =
\alpha_t \mathrm{d}  t+ \sqrt{2 \sigma} \mathrm{d} B_t$. The agent aims at minimizing
the expected cost \eqref{cost} where $\price(t)$ is the price of the
stock at time $t$. The agent is considered to be infinitesimal and has
no impact on $\price(t)$, so it assumes the price as given in its
optimization problem. On the other hand, in the equilibrium
configuration, the price $\price(t) \; (t \in [0,T])$ becomes
endogenous and indeed, is a function of the optimal behaviour of the
whole population of agents as formulated in $(iii)\eqref{MFGC}$.
The expression $D(t) := \int_{ \statespace} \phi(x,t) v(x,t) m(x,t) \;
\mathrm{d} x$ can be considered as a weighted net
demand formulation and the relation $P = \Psi (D)$ is the result of
supply-demand relation which determines the price of the good at the
market.
Concerning the role of the mapping $\phi$, one can think for example to the case of two exchangeable goods,
    i.e. $x \in \R^2$, with a price given by
  $P(t)= \Psi( \int_{\statespace} (\phi_1(x,t)v_1(x,t)+\phi_2(x,t) v_2(x,t)) m(x,t) \, \text{d} x )$, where $\Psi \colon
  \R \rightarrow \R$. The use of a mapping $\phi$, which is
  valued in $\R^{1 \times 2}$ and whose values depend on the scale
  chosen for the goods, is in such a situation necessary.
Thus, the system \eqref{MFGC} captures an equilibrium configuration. Similar models have been proposed in the electrical engineering literature, see for example \cite{ABTM18,CPTD12,PAS16} and the references therein.

In most mean field game models, the individual players interact through their position only, that is, via the variable $m$. The problem that we consider belongs to the more general class of problems, called \emph{extended mean field games}, for which the players interact through the joint probability distribution $\mu$ of states and controls.
Several existence results have been obtained for such models: in \cite{GPV14} for stationary mean field games, in \cite{GV16} for deterministic mean field games. In \cite[Section 5]{CL17}, a class of problems where $\mu$ enters in the drift and the integral cost of the agents is considered. We adopt the terminology \emph{mean field games of controls} employed by the authors of the latter reference. 
Let us mention that our existence proof is different from the one of \cite{CL17}, which includes control bounds.
In \cite[Section 1]{BLL18}, a model where the drift of the players depends on $\mu$ is analyzed.
In \cite{GS18}, a mean field game model is considered where at all time $t$, the average control (with respect to all players) is prescribed. 
We finally mention that extended mean field games have been studied with a probabilistic approach in \cite{ABVC18,CL15} and in \cite[Section 4.6]{CD18}, and that a class of linear-quadratic extended mean field games has been analyzed in \cite{PW18}.

A difficulty in the study of mean field games of controls, directly related to the supply-demand relation mentioned above, is the fact that the control variable, at a given time $t$, cannot be expressed in an explicit fashion as a function of $m(\cdot,t)$ and $u(\cdot,t)$. Instead, one has to analyze the well-posedness and the stability of a fixed point equation (see for example \cite[Lemma 5.2]{CL17}).
In our model, if we combine $(iii)$ and $(iv)$\eqref{MFGC}, we obtain the fixed point equation
\begin{equation} \label{eq:fixed_point}
v= -H_p(\nabla u + \Psi(\smallint \phi v m ))
\end{equation}
for the control variable $v$. A central idea of the present article is the following: equation \eqref{eq:fixed_point} is equivalent to the optimality conditions of a convex optimization problem, when $L$ is convex and $\Psi$ is the gradient of a convex function $\Phi$. This observation allows to show the existence and uniqueness of a solution $v$ (to equation \eqref{eq:fixed_point}) and to investigate its dependence with respect to $\nabla u$ and $m$ in a natural way. More precisely, we prove that this dependence is locally H\"older continuous.

The existence of a classical solution of \eqref{MFGC} is established with the Leray-Schauder theorem and classical estimates for parabolic equations. A similar approach has been employed in \cite{graber2015existence}, \cite{graber2017variational}, and \cite{graber2019} for the analysis of a mean field game problem proposed by Chan and Sircar in \cite{chan2017fracking}. In this model, each agent exploits an exhaustible resource and fixes its price. The evolution of the capacity of a given producer depends on the price set by the producer, but also on the average price (with respect to all producers).

The application of the Leray-Schauder theorem relies on a priori
bounds for fixed points. These bounds are obtained in particular with
a potential formulation of the mean field game problem: we prove that
all solutions to \eqref{MFGC} are also solutions to an optimal control
problem of the Fokker-Planck equation. We are not aware of any other
publication making use of such a potential formulation for a mean
field game of controls, with the exception of
\cite{graber2017variational} for the Chan and Sircar model. Let us
mention that besides the derivation of a priori bounds, the potential
formulation of the problem can be very helpful for the numerical
resolution of the problem and the analysis of learning procedures
(which are out of the scope of the present work).

The article is structured as follows. We list in Section \ref{section:assumption} the assumptions employed all along. The main result (Theorem \ref{theo:main}) is stated in Section \ref{section:result}. We provide in Section \ref{section:potential} a first incomplete potential formulation of the problem, incomplete in so far as the term $f(m)$ is not integrated.
We also introduce some auxiliary mappings, which allow to express $P$ and $v$ as functions of $m$ and $u$. We give some regularity properties for these mappings in Section \ref{section:auxiliary}. In Section \ref{section:apriori} we establish some a priori bounds for solutions to the coupled system. We prove our main result in Section \ref{section:leray-schauder}. In Section \ref{section:duality}, we give a full potential formulation of the problem, prove the uniqueness of the solution to \eqref{MFGC} and prove that $(u,P,f(m))$ is the solution to an optimal control problem of the HJB equation, under an additional monotonicity condition on $f$. Some parabolic estimates, used all along the article, are provided and proved in the appendix.

\section{Assumptions on data} \label{section:assumption}


Let us introduce the main notation used in the article.
Recall that $\densities$ was defined in \eqref{eq:densities}.
For all $m \in \densities$, for all measurable functions $v \colon \statespace \rightarrow \R^d$ such that $|v(\cdot)|^2 m(\cdot)$ is integrable, the following inequality holds true,
\begin{equation} \label{eq:jensen}
\Big| \int_{\statespace} v(x) m(x) \; \mathrm{d} x \Big|^2 \leq \int_{\statespace} |v(x)|^2 m(x) \; \mathrm{d} x,
\end{equation}
by the Cauchy-Schwarz inequality.

The gradient of the data functions with respect to some variable is denoted with an index, for example, $H_p$ denotes the gradient of $H$ with respect to $p$. The same notation is used for the Hessian matrix. The gradient of $u$ with respect to $x$ is denoted by $\nabla u$. Let us mention that very often, the variables $x$ and $t$ are omitted, to alleviate the calculations. We also denote by $\int \phi v m$ the integral $\int_{\statespace} \phi v m \; \mathrm{d} x$ when used as a second argument of $\Psi$. For a given normed space $X$, the ball of center 0 and radius $R$ is denoted $B(X,R)$.

Along the article, we use the following H\"older spaces:
$\mathcal{C}^\alpha(Q)$, $\mathcal{C}^{2+ \alpha}(\statespace)$, and
$\mathcal{C}^{2+ \alpha,1+\alpha/2}(Q)$, defined as usual with $\alpha
\in (0,1)$. Sobolev spaces are denoted by $W^{k,p}$, the order of
derivation $k$ being possibly non-integral (see their definition in \cite[section II.2]{LSU}).
We fix now a real number $p$ such that
\begin{equation*}
p > d+2.
\end{equation*}
We will also make use of the following Banach space:
\begin{equation*}
W^{2,1,p}(Q)= L^p(0,T;W^{2,p}(\statespace)) \cap W^{1,p}(Q).
\end{equation*}

\paragraph{Convexity assumptions}

We collect below the required assumptions on the data. As announced in the introduction, $H$ is related to the convex conjugate of a mapping $L \colon Q \times \R^d \rightarrow \R$ as follows:
\begin{equation} \label{eq:conjugate}
H(x,t,p)= L^*(x,t,-p)= \sup_{v \in \R^d} - \langle p,v \rangle - L(x,t,v).
\end{equation}
The mapping $L$ is assumed to be strongly convex in its third variable, uniformly in $x$ and $t$, that is, we assume that $L$ is differentiable with respect to $v$ and that there exists $C> 0$ such that
\begin{equation*}
\label{eq:grad_monotony} \tag{A1}
\langle L_v(x,t,v_2) -L_v(x,t,v_1), v_2 -v_1 \rangle \geq \frac{1}{C} |v_2- v_1|^2,
\end{equation*}
for all $(x,t) \in Q$ and for all $v_1$ and $v_2 \in \R^d$. This ensures that $H$ takes finite values and that $H$ is continuously differentiable with respect to $p$, as can be easily checked. Moreover, the supremum in \eqref{eq:conjugate} is reached for a unique $v$, which is then given by $v= -H_p(x,t,p)$, i.e.
\begin{equation} \label{eq:v_eq_minus_Hp}
H(x,t,p)+ L(x,t,v) + \langle p,v \rangle = 0
\Longleftrightarrow
v= -H_p(x,t,p),
\end{equation}
for all $(x,t,p,v) \in Q \times \R^d \times \R^d$.

We also assume that $\Psi$ has a potential, that is, there exists a mapping $\Phi \colon [0,T] \times \R^k \rightarrow \R$, differentiable in its second argument, such that
\begin{equation} \label{eq:Phi_convex}
\Psi(t,z)= \Phi_z(t,z), \quad \forall (t,z) \in [0,T] \times \R^k.
\end{equation}

\paragraph{Regularity assumptions}

We assume that $L_v$ is differentiable with respect to $x$ and $v$ and that $\phi$ is differentiable with respect to $x$.
All along the article, we make use of the following assumptions. \vspace{0.2mm}

\noindent \emph{Growth assumptions} There exists $C>0$ such that for all $(x,t) \in Q$, $y \in \statespace$, $v \in \R^d$, $z \in \R^k$, and $m \in \densities$,
\begin{align*}
& \bullet \quad L(x,t,v) \leq C |v|^2 + C \label{ass_L_quad_growth2} \tag{A2} \\
& \bullet \quad |L(y,t,v)-L(x,t,v)| \leq C |y-x| (1 + |v|^2) \label{ass_L_Lipschitz} \tag{A3} \\
& \bullet \quad |\Psi(t,z)| \leq C |z| + C \label{ass_Psi_lin_growth} \tag{A4} \\
& \bullet \quad |f(x,t,m)| \leq C. \label{ass_f_bounded} \tag{A5} \\
\intertext{\emph{H\"older continuity assumptions}}
& \bullet \quad \text{For all $R>0$, there exists $\alpha \in (0,1)$ such that} \\
& \qquad
\begin{cases}
\begin{array}{rl}
L \in & \mathcal{C}^{\alpha}(B_R), \\
L_v \in & \mathcal{C}^{\alpha}(B_R, \R^d), \\
L_{vx} \in & \mathcal{C}^{\alpha}(B_R, \R^{d \times d}), \\
L_{vv} \in &  \mathcal{C}^{\alpha}(B_R, \R^{d \times d}),
\end{array}
\end{cases}
\quad
\begin{cases}
\begin{array}{rl}
\Psi \in & \mathcal{C}^{\alpha}(B_R',\R^d), \\
\phi \in & \mathcal{C}^{\alpha}(Q,\R^{k \times d}), \\
  D_x \phi \in & \mathcal{C}^{\alpha}(Q,\R^{k \times d \times d}),
\end{array}
\end{cases} \qquad \label{ass_holder} \tag{A6} \\[0.5em]
& \phantom{\bullet} \quad \ \, \text{where $B_R= Q \times B(\R^d,R)$ and $B_R'= [0,T] \times B(\R^k,R)$.} \\
& \bullet \quad \text{There exists $\alpha \in (0,1)$ and $C>0$ such that} \\
& \qquad \qquad | f(x_2,t_2,m_2) - f(x_1,t_1,m_1) | \\
& \qquad \qquad \qquad \leq C \big( |x_2 - x_1| + |t_2 - t_1|^\alpha + \| m_2 - m_1 \|_{L^{\infty}(\statespace)}^\alpha \big), \label{ass_hold_f} \tag{A7} \\
& \phantom{\bullet} \quad \ \, \text{for all $(x_1,t_1)$ and $(x_2,t_2) \in Q$ and for all $m_1$ and $m_2 \in \densities$.} \\
& \bullet \quad \text{There exists $\alpha \in (0,1)$ such that $m_0 \in \mathcal{C}^{2+ \alpha}(\statespace)$, $g \in \mathcal{C}^{2 +\alpha}(\statespace)$}. \label{ass_init_cond} \tag{A8}
\end{align*}

Let us mention here that the variables $C>0$ and $\alpha \in (0,1)$ used all along the article are generic constants. The value of $C$ may increase from an inequality to the next one and the value of the exponent $\alpha$ may decrease.

Some lower bounds for $L$ and for $\Phi$ can be easily deduced from
the convexity assumptions. By assumption \eqref{ass_holder},
$L(x,t,0)$ and $L_v(x,t,0)$ are bounded. It follows then from the
strong convexity assumption \eqref{eq:grad_monotony}
that there exists a constant $C> 0$ such that
\begin{equation} \label{L_quad_growth1}
  \frac{1}{C} |v|^2 - C \leq L(x,t,v), \quad
  \text{ for all }(x,t,v) \in Q \times \R^d. 
  \quad 
\end{equation}
Without loss of generality, we can assume that $\Phi(t,0)=0$, for all
$t \in [0,T]$. Since $\Phi$ is convex, we have that $\Phi(t,z) \geq
\langle \Psi(t,0), z \rangle$, for all $z \in \R^k$. We deduce then
from assumption \eqref{ass_Psi_lin_growth} that
\begin{equation} \label{eq:phi_linear_growth}
\Phi(t,z) \geq -C | z |, \quad \text{for all $z \in \R^k$},
\end{equation}
where $C$ is independent of $t$ and $z$.

Some regularity properties for the Hamiltonian can be deduced from the convexity assumption \eqref{eq:grad_monotony} and the H\"older continuity of $L$ and its derivatives (assumption \eqref{ass_holder}). They are collected in the following lemma.

\begin{lemma} \label{lemma:reg_H}
The Hamiltonian $H$ is differentiable with respect to $p$ and $H_p$ is differentiable with respect to $x$ and $p$. Moreover, for all $R>0$, there exists $\alpha \in (0,1)$ such that $H \in \mathcal{C}^{\alpha}(B_R)$, $H_p \in \mathcal{C}^{\alpha}(B_R, \R^d)$, $H_{px} \in \mathcal{C}^{\alpha}(B_R, \R^{d \times d})$, and $H_{pp} \in \mathcal{C}^{\alpha}(B_R, \R^{d \times d})$
\end{lemma}

\begin{proof}
For a given $(x,t,p) \in Q \times \R^d$, there exists a unique $v:= v(x,t,p)$ maximizing the function $v \in \R^d \mapsto -\langle p,v \rangle - L(x,t,v)$, which is strongly concave by \eqref{eq:grad_monotony}. It is then easy to deduce from \eqref{L_quad_growth1} and the boundedness of $L(x,t,0)$ that there exists a constant $C$, independent of $(x,t,p)$, such that
$|v(x,t,p)| \leq C( |p| +1)$.
For all $(x,t,p) \in Q \times \R^d$, we have
\begin{equation} \label{eq:easy_opti_cond}
p + L_v(x,t,v(x,t,p))= 0.
\end{equation}
Since $L_v$ is continuously differentiable with respect to $x$ and $v$, we obtain with the inverse mapping theorem that $v(x,t,p)$ is continuously differentiable with respect to $x$ and $p$.
Let $R > 0$ and let $(x_1,t_1,p_1)$ and $(x_2,t_2,p_2) \in Q \times B_R$. Let $v_i= v(x_i,t_i,p_i)$ for $i=1,2$. We have $|v_i| \leq C$, where $C$ does not depend on $x_i$, $t_i$, and $p_i$ (but depends on $R$).
Moreover, we have
\begin{align*}
& \langle p_2 - p_1, v_2 -v_1 \rangle + \langle L_v(x_2,t_2,v_2)-L_v(x_1,t_1,v_2), v_2 - v_1 \rangle \\
& \qquad \qquad \qquad \qquad \qquad + \langle L_v(x_1,t_1,v_2)-L_v(x_1,t_1,v_1), v_2- v_1 \rangle = 0.
\end{align*}
We deduce from \eqref{eq:grad_monotony}, Young's inequality, and \eqref{ass_holder} that there exists $C>0$ and $\alpha \in (0,1)$, both independent of $x_i$, $t_i$, and $p_i$ such that 
\begin{align*}
\frac{1}{C} |v_2 - v_1|^2
\leq \ & | \langle p_2 - p_1, v_2 -v_1 \rangle | + | \langle L_v(x_2,t_2,v_2)-L_v(x_1,t_1,v_2), v_2 - v_1 \rangle | \\
\leq \ & \frac{1}{2 \varepsilon} |p_2 - p_1|^2 + \varepsilon |v_2 - v_1|^2 + \frac{C}{\varepsilon} \big( |x_2 - x_1|^\alpha + |t_2 - t_1|^\alpha \big),
\end{align*}
for all $\varepsilon > 0$. Taking $\varepsilon = \frac{1}{2C}$, we deduce that the mapping $(x,t,p) \in B_R \mapsto v(x,t,p)$ is H\"older continuous. Since $L$ is H\"older continuous on bounded sets, we obtain that the Hamiltonian $H(x,t,p)= -\langle p,v(x,t,p) \rangle - L(x,t,v(x,t,p))$ is H\"older continuous on $B_R$.

One can easily check that $H_p(x,t,p)= -v(x,t,p)$, which proves that $H_p$ is H\"older continuous on $B_R$. Finally, differentiating relation \eqref{eq:easy_opti_cond} with respect to $x$ and $p$, we obtain that
\begin{align*}
D_x v(x,t,p)= \ & - L_{vv}(x,t,v(x,t,p))^{-1} L_{vx}(x,t,v(x,t,p)) \\
D_p v(x,t,p)= \ & - L_{vv}(x,t,v(x,t,p))^{-1}.
\end{align*}
We deduce then with assumption \eqref{ass_holder} that $D_x v(x,t,p)$ and $D_p v(x,t,p)$ (and thus $H_{px}$ and $H_{pp}$) are H\"older continuous on $B_R$, as was to be proved.
\qed \end{proof}

\paragraph{An example of coupling term} 

We finish this section with an example of a mapping $f$ satisfying the regularity assumptions \eqref{ass_f_bounded} and \eqref{ass_hold_f}.
Let $\varphi \in L^\infty(\R^d)$
be a given Lipschitz continuous mapping, with modulus $C_1$.
Let us set 
  $C_2= \| \varphi \|_{L^\infty(\R^d)}$.
Let $K \colon Q \times [-C_2,C_2] \rightarrow \R$ be a measurable mapping satisfying the following assumptions:
\begin{enumerate}
\item The mapping $x \in \statespace \mapsto K(x,0,0)$ lies in $L^1(\statespace)$.
\item There exist a mapping $C_3 \in L^1(\statespace)$ and $\alpha \in (0,1)$ such that for a.e.\@ $x \in \statespace$, for all $t_1$ and $t_2 \in [0,T]$ and for all $w_1$ and $w_2 \in [-C_2,C_2]$,
\begin{equation*}
|K(x,t_2,w_2)-K(x,t_1,w_1)|
\leq C_3(x) \big( |t_2-t_1|^\alpha + |w_2 - w_1|^\alpha \big).
\end{equation*}
\end{enumerate}
Let us set $\tilde{\varphi}(x) := \varphi(-x)$.
We identify $m\in L^\infty(\statespace)$ with its extension by 0 over
$\R^d$ 
so that the convolution product below is well-defined:
\be
m*\varphi (x) :=\int_{\R^d} m(x-y) \varphi(y) \; \mathrm{d} y,
\;\; x \in\statespace.
\ee
We keep in mind that
$m*\varphi$ is a function over $\statespace$.
Then
\be
\label{ets-m-varphi}
\| m*\varphi \|_{L^\infty(\statespace)} \leq \| \varphi \|_{L^\infty(\statespace)} = C_2,
\quad \text{for all $m\in  \densities$.}
\ee
In a similar way we can define
\be
f_K(x,t,m)= (K(\cdot,t,m * \varphi(\cdot)) * \tilde{\varphi})(x),
\ee
and we have that
\begin{align}
& \| f_K(x,t,m) \|_{L^\infty(\statespace)}
\leq \| K(\cdot,t,m * \varphi) \|_{L^1(\statespace)} \|\tilde{\varphi} \|_{L^\infty(\statespace)} \notag \\
& \quad \leq ( \| K(\cdot,0,0) \|_{L^1(\statespace)} + \| C_3 \|_{L^1(\statespace)}(T^\alpha + \| {\varphi} \|_{L^\infty(\statespace)}) ) \| \tilde{\varphi} \|_{L^\infty(\statespace)}. \label{lemma:ex_f-1}
\end{align}
The specific structure of $f_K$ is actually motivated by the fact that under an additional monotonicity assumption, $f_K$ derives from a potential (as proved in \cite[Example 1.1]{Cdga13}). For the moment, we have the following regularity result.

\begin{lemma} \label{lemma:ex_f}
The above mapping $f_K$
satisfies assumptions \eqref{ass_f_bounded} and \eqref{ass_hold_f}.
\end{lemma}

\begin{proof}
  Assumption \eqref{ass_f_bounded}
  follows from \eqref{lemma:ex_f-1}.
  We next prove \eqref{ass_hold_f}.
  Let $(x_1,t_1)$ and $(x_2,t_2) \in Q$, let
  $m_1$ and $m_2 \in \densities$.
  Then
\begin{align*}
& |f_K(x_2,t_2,m_2)-f_K(x_1,t_2,m_2)| \\
& \qquad \leq
           \|K(\cdot,t_2,m_2 * \varphi(\cdot)) \|_{L^\infty(\statespace)}
           \|\varphi (x_2-\cdot)- \varphi(x_1-\cdot) \|_{L^\infty(\statespace)} \\
& \qquad \leq C_1 C | x_2 - x_1 |.
\end{align*}
Also, 
\begin{align*}
& |f_K(x_1,t_2,m_2)-f_K(x_1,t_1,m_1)| \\
& \qquad \leq \|K(\cdot,t_2,m_2 * \varphi(\cdot))-K(\cdot,t_1,m_1 * \varphi(\cdot)) \|_{L^1(\statespace)}
        \|\varphi\|_{L^\infty(\statespace)}\\
& \qquad \leq C_2 \| C_3 \|_{L^1(\statespace)} \big( |t_2-t_1|^\alpha + \| (m_2 - m_1) * \varphi \|_{L^\infty(\statespace)}^\alpha \big).
\end{align*}
Finally, we have $\| (m_2 - m_1) * \varphi \|_{L^\infty(\statespace)} \leq \| m_2 - m_1 \|_{L^\infty(\statespace)} \| \varphi \|_{L^\infty(\statespace)}$ and thus, assumption \eqref{ass_hold_f} follows.
\qed \end{proof}

\section{Main result and general approach} \label{section:result}

\begin{theorem} \label{theo:main}
There exists $\alpha \in (0,1)$ such that \eqref{MFGC} has a classical solution $(u,m,v,P)$, with
\begin{equation} \label{eq:main}
\begin{cases}
\begin{array}{rl}
m \in & \mathcal{C}^{2+\alpha,1+ \alpha/2}(Q), \\
u \in & \mathcal{C}^{2+\alpha,1+ \alpha/2}(Q), \\
P \in & \mathcal{C}^\alpha(0,T;\R^k), \\
v \in & \mathcal{C}^{\alpha}(Q,\R^d), \ D_x v \in \mathcal{C}^{\alpha}(Q,\R^{d \times d}).
\end{array}
\end{cases}
\end{equation}
\end{theorem}

The result is obtained with the Leray-Schauder theorem, recalled below.

\begin{theorem}[Leray-Schauder]
Let $X$ be a Banach space and let $\mathcal{T} \colon X \times [0,1] \to X$ be a continuous and compact mapping. Let $x_0 \in X$. Assume that $\mathcal{T}(x,0)=x_0$ for all $x \in X$ and assume there exists $C>0$ such that $\norm x \norm_X < C$ for all $(x,\tau) \in X \times [0,1]$ such that $\mathcal{T} (x,\tau) = x$. Then, there exists $x \in X$ such that $\mathcal{T}(x,1) = x$.
\end{theorem}

A proof of the theorem can be found in \cite[Theorem 11.6]{gilbarg2015elliptic}, for $x_0=0$. The extension to a general value of $x_0$ can be easily obtained with a translation argument that we do not detail.
The application of the Leray-Schauder theorem and the construction of
$\mathcal{T}$ will be detailed in Section
\ref{section:leray-schauder}.
Let us mention that the set of fixed points of $\mathcal{T}(\cdot,\tau)$,
for $\tau \in [0,1]$, 
will coincide with the set of solutions of the following parametrization of \eqref{MFGC}:
\begin{equation}\label{MFGC_tau}
\left\lbrace 
\begin{array}{l l l}
(i) \quad &-\partial_t u -\sigma \Delta u + \tau H( \nabla u
            + \phi^\intercal \price (t) ) = \tau f(m(t)) \quad &(x,t) \in Q, \\[5pt]
(ii) \quad & \partial_t m - \sigma \Delta m + \tau \mathrm{div} (mv)= 0 \quad &(x,t) \in Q, \\[5pt]
(iii) \quad & \price(t) = \Psi \left(t, \int_{ \statespace } \phi(x,t) v(x,t) m(x,t) \; \mathrm{d} x \right) \quad & t \in [0,T],\\[5pt]
(iv) \quad & v(x,t) =  - H_p(x,t, \nabla u (x,t) + \phi(x,t)^\intercal\price(t) ) &(x,t) \in Q, \\[5pt]
(v) \quad & m(x,0)= m_0(x), \quad u(x,T) = \tau g(x) \quad & x
                                                                  \in
                                                                  \statespace,\\[5pt]
\end{array} 
\tag{MFGC$_\tau$} \right.
\end{equation}
Of course, \eqref{MFGC_tau} corresponds to \eqref{MFGC} for $\tau= 1$.
Let us introduce the spaces $X$ and $X'$, used for the formulation of the fixed-point equation:
\begin{equation*}
X := \big( W^{2,1,p}(Q) \big)^2, \quad
X' := X \times L^\infty(Q,\R^d) \times L^\infty(0,T;\R^k).
\end{equation*}
The HJB equation $(i)$ and the Fokker-Planck
equation $(ii)$ are classically understood in the viscosity and weak
sense, respectively. However, due to the choice of the solution spaces,
we may interpret these equations as equalities in $L^p(Q)$: in particular, if $u \in W^{2,1,p}(Q)$ and $P \in L^\infty(0,T;\R^k)$, we have that $\nabla u \in L^\infty(Q;\R^d)$ (by Lemma \ref{lemma:max_reg_embedding}), and thus $H(\nabla u + \phi^\intercal P(t)) \in L^\infty(Q)$.
A first and important step of our analysis is the construction of auxiliary mappings allowing to express $v$ and $P$ as functions of $m$ and $u$. These mappings cannot be obtained in a straightforward way, since in $(iii)$, $P$ depends on $v$ and in $(iv)$, $v$ depends on $P$. 

\begin{lemma} \label{lemma:max_principle}
Let $\tau \in [0,1]$, let $(m,v) \in W^{2,1,p}(Q) \times L^\infty(Q,\R^d)$ be a weak solution to the Fokker-Planck equation $\partial_t m - \sigma \Delta m + \tau \text{\emph{div}}(vm)= 0$, $m(\cdot,0)= m_0(\cdot)$.
Then $m \geq 0$ and for all $t \in [0,T]$, $\int_{\statespace} m(x,t) \; \mathrm{d} x = 1$.
\end{lemma}

\begin{proof}
  Multiply (MFGC$_\tau$)(ii) by $\mu(x,t):=\min(0,m(x,t))$.
  Use $\nabla\mu(x,t)=\1B_{\{m(x,t)<0\}} \nabla m(x,t)$, so that
  integrating (by parts) over $Q_t:=\statespace\times (0,t)$,
  since $v$ is essentially bounded, we get that
  \begin{align*}
  & \frac{1}{2} \int_{\statespace} \mu(x,t)^2 \; \mathrm{d} x
  + \sigma \iint_{Q_t}  | \nabla \mu(x,s)|^2 \; \mathrm{d} x \; \mathrm{d} s
  =
  \tau \iint_{Q_t} \langle v, \nabla \mu \rangle m \; \mathrm{d} x \; \mathrm{d} s \\
  & \qquad =
  \tau \iint_{Q_t} \langle v,\nabla \mu \rangle \mu \; \mathrm{d} x \; \mathrm{d} s
  \leq  C \iint_{Q_t} |\mu|^2 \; \mathrm{d} x \; \mathrm{d} s
  + \sigma \iint_{Q_t}  | \nabla \mu|^2 \; \mathrm{d} x \; \mathrm{d} s,
  \end{align*}
  so that after cancellation of the contribution of
  $\nabla\mu$, we obtain, applying
  Gronwall's lemma to $a(t):=\int_{\statespace} \mu(x,t)^2$,
  that $a(t)=0$ for all $t$ which means that $m$ is
  non-negative.
  Moreover, for all $t \in [0,T]$,
\begin{equation*}
\int_{\statespace} m(x,t) \; \mathrm{d} x
= \int_{\statespace} m(x,0) \; \mathrm{d} x + \iint_{Q_t} \sigma \Delta m - \tau \text{div}(v m) \; \mathrm{d} x \; \mathrm{d} s.
\end{equation*}
Integrating by parts the double integral we see that it is equal to 0,
and we conclude by noting that
$\int_{\statespace} m(x,0) \; \mathrm{d} x = \int_{\statespace} m_0(x) \; \mathrm{d} x = 1$.
\qed \end{proof}

\section{Potential formulation} \label{section:potential}

In this section, we first establish a potential formulation of the mean field game problem \eqref{MFGC_tau}, that is to say, we prove that for $(u_\tau,m_\tau,v_\tau,P_\tau) \in X'$ satisfying \eqref{MFGC_tau}, $(m_\tau,v_\tau)$ is a solution to an optimal control problem. We prove then that for all $t$, $v_\tau(\cdot,t)$ is the unique solution of some optimization problem, which will enable us to construct the announced auxiliary mappings.

Let us introduce the cost functional $B \colon W^{2,1,p}(Q) \times L^\infty(Q,\R^d) \times L^\infty(Q) \rightarrow \R$, defined by
\begin{align}
& B (m,v;\tilde{f}) =
\iint_Q \big( L ( x,t, v(x,t)) + \tilde{f}(x,t) \big) m(x,t) \; \mathrm{d} x \; \mathrm{d} t \notag \\
& \qquad + \int_0^T \Phi \Big(t, \int_{\statespace} \phi(x,t) v(x,t)  m(x,t) \; \mathrm{d} x \Big) \, \mathrm{d} t + \int_{\statespace} g(x) m(x,T) \; \mathrm{d} x. \label{potential}
\end{align}
We have the following result. 

\begin{lemma} \label{lemma:var_princ_1}
For all $\tau \in [0,1]$ and $(u_\tau,m_\tau,v_\tau,P_\tau) \in X'$ satisfying \eqref{MFGC_tau}, the pair $(m_\tau,v_\tau)$ is the solution to the following optimization problem:
\begin{equation} \label{eq:primal_problem}
\min_{\begin{subarray}{c} m \in W^{2,1,p}(Q) \\ v \in
    L^\infty(Q,\R^d) \end{subarray}} \
B(m,v;\tilde{f}_\tau), \
\text{s.t.: }
\begin{cases}
\begin{array}{rl}
\partial_t m - \sigma \Delta m + \tau \mathrm{div} (vm)= & 0, \\
m(x,0)= & m_0(x),
\end{array}
\end{cases}
\end{equation}
where $\tilde{f}_\tau(x,t)= f(x,t,m_\tau(t))$.
\end{lemma}

\begin{remark}
Let us emphasize that the above optimal control problem is only an incomplete potential formulation, since the term $\tilde{f}_\tau$ still depends on $m_\tau$.
\end{remark}

\begin{proof}[Lemma \ref{lemma:var_princ_1}]
Let us consider the case where $\tau \in (0,1]$.
Let $(m,v) \in W^{2,1,p}(Q) \times L^\infty(Q,\R^d)$
be a feasible pair, i.e., it satisfies the constraint in \eqref{eq:primal_problem}. 
For all $(x,t) \in Q$, we have $v_\tau = - H_p(\nabla u_\tau + \phi^\intercal P_\tau)$. Therefore, by \eqref{eq:conjugate} and \eqref{eq:v_eq_minus_Hp}, we have that
\begin{align*}
L (v) \geq \ & - H(\nabla u_\tau + \phi^\intercal \price_\tau ) - \myprod{\nabla u_\tau + \phi^\intercal \price_\tau}{v},  \\
L (v_\tau ) = \ & - H(\nabla u_\tau + \phi^\intercal \price_\tau ) - \myprod{\nabla u_\tau + \phi^\intercal \price_\tau }{v_\tau},
\end{align*}
for all $(x,t) \in Q$.
Moreover, by Lemma \ref{lemma:max_principle}, $m \geq 0$ and $m_\tau \geq 0$.
Therefore,
\begin{align} 
& L (v) m - L (v_\tau) m_\tau \notag \\
& \qquad \geq - H(\nabla u_\tau + \phi^\intercal \price_\tau) (m - m_\tau) -\myprod{\nabla u_\tau + \phi^\intercal \price_\tau }{vm - v_\tau m_\tau }. \label{eq:form_var1}
\end{align}
Using $(i)$\eqref{MFGC_tau}, we obtain
\begin{align*}
& L (v) m - L (v_\tau) m_\tau \\
& \qquad \geq
\frac{1}{\tau}(-\partial_t u_\tau -\sigma \Delta u_\tau- \tau \tilde{f}_\tau) ( m -
m_\tau )
- \myprod{\nabla u_\tau + \phi^\intercal \price_\tau }{v m - v_\tau m_\tau }.
\end{align*}
After integration with respect to $x$, we obtain that for all $t$,
\begin{align*}
& \int_{\statespace} (L (v) m - L (v_\tau) m_\tau) + \tilde{f}_\tau(m-m_\tau) \; \mathrm{d} x \notag \\
  & \qquad \geq \frac{1}{\tau} \int_{\statespace}
    (-\partial_t u_\tau -\sigma \Delta u_\tau) ( m - m_\tau ) \; \mathrm{d} x
- \int_{\statespace} \myprod{\nabla u_\tau}{v m - v_\tau m_\tau } \; \mathrm{d} x \\
& \qquad  \qquad - \langle P_\tau, \smint \phi(v m - v_\tau m_\tau) \rangle.
\end{align*}
We obtain with the convexity of $\Phi$ and $(iii)$\eqref{MFGC_tau} that
\begin{align}
\Phi (\smint \phi m v) - \Phi (\smint \phi v_\tau m_\tau)
\geq \ & \langle \Psi (\smint \phi m_\tau v_\tau), \smint \phi (v m - v_\tau m_\tau) \rangle \notag \\
= \ & \langle P_\tau, \smint \phi (m v - m_\tau v_\tau) \rangle. \label{eq:form_var2}
\end{align}
Using the previous calculations to bound $B(m,v;\tilde{f}_\tau)-B(m_\tau,v_\tau;\tilde{f}_\tau)$
from below, we observe that the
term $\langle P_\tau, \smint \phi (m - m_\tau v_\tau) \rangle$
cancels out and obtain
\begin{align*}
& B (m,v;\tilde{f}_\tau) - B (m_\tau,v_\tau;\tilde{f}_\tau) \\
& \qquad \geq \iint_Q \frac{1}{\tau}(-\partial_t u_\tau -\sigma \Delta u_\tau) (m-m_\tau)
         - \myprod{\nabla u_\tau}{m v - m_\tau v_\tau} \; \mathrm{d} x \; \mathrm{d} t \\
& \qquad \qquad + \int_{\statespace} g(x) (m(x,T) - m_\tau(x,T) ) \; \mathrm{d} x.
\end{align*}
Integrating by parts and using $(ii)$\eqref{MFGC_tau}, we finally obtain that
\begin{equation*}
B (m,v;\tilde{f}_\tau) - B (m_\tau,v_\tau;\tilde{f}_\tau) \geq \frac{1}{\tau} \int_{\statespace} u_\tau(0,x)( m(0,x)-m_\tau(0,x) )\; \mathrm{d} x = 0,
\end{equation*}
as was to be proved.
We do not detail the proof for the case $\tau=0$, which is actually simpler. Indeed, for $\tau=0$, the solution to the Fokker-Planck equation is independent of $v$ and thus $m= m_\tau$ in the above calculations.
\qed \end{proof}

We have proved that the pair $(m_\tau,v_\tau)$ is the solution to an
optimal control problem. Therefore, for all $t$, $v_\tau(\cdot,t)$
minimizes the Hamiltonian associated with problem
\eqref{eq:primal_problem}. Let us introduce some notation, in order to
exploit this property. For $m \in \densities$, we denote by
$L_m^2(\statespace,\R^d)$ the Hilbert space of measurable mappings $v
\colon \statespace \rightarrow \R^d$ such that $\int_{\statespace}
|v|^2 m < \infty$, equipped with the scalar product
$\int_{\statespace} \langle v_1,v_2 \rangle m$.
An element of $L_m^2(\statespace)$ is an
equivalent class of functions equal  $m$-almost everywhere.
Note that $L^\infty(\statespace) \subset L_m^2(\statespace)$.

For $t \in [0,T]$, $m \in \densities$, and
$w \in L^\infty(\statespace,\R^d)$,
we consider the mapping
\begin{equation*}
v \in L_m^2(\statespace, \R^d) \mapsto
J(v;t,m,w):= \Phi \big( t, \smint \phi v m \big) + \int_{\statespace} \big( L(v)+ \langle w,v \rangle \big) m \; \mathrm{d} x.
\end{equation*}
Combining inequalities \eqref{eq:form_var1} and \eqref{eq:form_var2} (with $m=m_\tau$),
we directly obtain that for all $t \in [0,T]$, for all $v \in L_{m}^2(\statespace,\R^d)$ with $m=m_\tau(\cdot,t)$,
\begin{equation*}
J \big( v;t,m_\tau(t),\nabla u_\tau(t) \big) \geq J \big( v_\tau(t);t,m_\tau(t),\nabla u_\tau(t) \big).
\end{equation*}
The following lemma will enable us to express $P_\tau(t)$ and $v_\tau(\cdot,t)$ as functions of $m_\tau(\cdot,t)$ and $u_\tau(\cdot,t)$. The key idea is, roughly speaking, to prove the existence and uniqueness of a minimizer to $J(\cdot;t,m,w)$.

\begin{lemma} \label{lemma:stat_existence}
For all $t \in [0,T]$, for all $m \in \densities$, for all $R > 0$, and for all $w \in L^\infty(\statespace,\R^d)$ such that $\| w \|_{L^\infty(\statespace,\R^d)} \leq R$, there exists a unique pair $(v,P) \in L^\infty(\statespace,\R^d) \times \R^k$,  such that
\begin{equation} \label{eq:v_equal_H_p_0}
\begin{cases}
\begin{array}{rl}
v(x)= & - H_p (x,t, w(x) +  \phi(x,t)^\intercal P), \quad \text{for a.e. $x \in \statespace$}, \\
P= & \Psi(t, \smint \phi v m).
\end{array}
\end{cases}
\end{equation}
The pair $(v,P)$ is then denoted $(\mathbf{v}(t,m,w),\mathbf{P}(t,m,w))$.
Moreover, we have
\begin{equation} \label{eq:bound_v_P}
\| \mathbf{v}(t,m,w) \|_{L^\infty(\statespace,\R^d)} \leq C \quad \text{and} \quad |\mathbf{P}(t,m,w)| \leq C,
\end{equation}
where the constant $C$ is independent of $t$, $m$, and $w$ (but depends on $R$).
\end{lemma}

\begin{proof}
If the pair $(v,P)$ satifies \eqref{eq:v_equal_H_p_0}, then
\begin{equation} \label{eq:v_equal_H_p}
  v= - H_p ( w + \phi^\intercal \Psi(\smint \phi v m) )
  \quad \text{a.e. on $\statespace$}.
\end{equation}
One can easily check that for proving the existence and uniqueness of a pair $(v,P)$ satisfying \eqref{eq:v_equal_H_p_0}, it is sufficient to prove the existence and uniqueness of $v \in L^\infty(\statespace,\R^d)$ satisfying \eqref{eq:v_equal_H_p}. For future reference, let us observe that by \eqref{eq:v_eq_minus_Hp}, relation \eqref{eq:v_equal_H_p} is equivalent to
\begin{equation} \label{eq:opti_cond_v}
\phi^\intercal(x,t) \Psi(t,\smint \phi v m) + L_v(x,t,v(x)) + w(x)= 0, \quad \text{for a.e. $x \in \statespace$}.
\end{equation}

\emph{Step 1}: existence and uniqueness of a minimizer of
$J(\cdot;t,m,w)$.\\
In view of (A1),
$v \mapsto \int_{\statespace} L(v)m \; \mathrm{d} x$
is strongly convex over $L_m^2(\statespace,\R^d)$.
Since the sum of a l.s.c.\@ convex function 
and of a l.s.c.\@ strongly convex function is l.s.c.\@ and strongly
convex,
so is the function $J(\cdot;t,m,w)$.
Thus, it possesses a unique minimizer $\bar{v}$ in
$L_m^2(\statespace,\R^d)$. We obtain
\begin{equation}
  \label{eq:estim_bar_v} 
C\|\bar v\|^2_{L_m^2(\statespace,\R^d)} - C \leq J(\bar v;t,m,w)) \leq J(0;t,m,w))  =C,
\end{equation}
so that
$\|\bar v\|^2_{L_m^2(\statespace,\R^d)} \leq C$, with
$C$ independent of $t$, $m$, and $w$, but depending on $R$, as all constants $C$ used in the proof.

\emph{Step 2:} existence of $\mathbf{v}(t,m,w)$ and a priori bound. \\
One can check that the mapping $\delta v \in L^\infty(\statespace,\R^d) \mapsto J(\bar{v}+ \delta v;t,m,w)$ is differentiable.
Since $\bar{v}$ is optimal, the derivative of the above mapping is null at $\delta v=0$ and thus
\begin{equation*}
\big[ \phi^\intercal(x,t) \Psi(t,\smint \phi(x')\bar v(x')m(x')\, \mathrm{d} x') + L_v(x,t,\bar v(x)) + w(x) \big] m(x)= 0,
\end{equation*}
for a.e.\@ $x \in \statespace$.
Using then the equivalence of \eqref{eq:v_equal_H_p} and \eqref{eq:opti_cond_v}, we obtain that
\begin{equation*}
m(x) > 0 \Longrightarrow
\bar{v}(x)= - H_p \big(x,t,w(x) + \phi^\intercal(x,t) \Psi(t, \smint \phi(x',t) \bar{v}(x')m(x') \, \mathrm{d} x') \big),
\end{equation*}
for a.e.\@ $x \in \statespace$.
Consider now the measurable function $v$ defined by
\begin{equation*}
v(x)= - H_p \big(x,t,w(x) + \phi^\intercal(x,t) \Psi(t,\smint \phi(x',t) \bar{v}(x') m(x')\, \mathrm{d} x'\big),
\end{equation*}
for a.e.\@ $x \in \statespace$.
The two functions $v$ and $\bar{v}$ may not be equal for a.e.\@ $x$ if $m(x)=0$ on a subset of $\statespace$ of non-zero measure. Still they are equal in $L_m^2(\statespace,\R^d)$, which ensures in particular that $\smint \phi(x',t) \bar{v}(x') m(x') \, \mathrm{d} x' = \smint \phi(x',t) {v}(x') m(x') \, \mathrm{d} x'$ and finally that $v$ satisfies \eqref{eq:v_equal_H_p} and lies in $L^\infty(\statespace,\R^d)$, as a consequence of the continuity of $H_p$ (proved in Lemma \ref{lemma:reg_H}). We also have that $\| \bar{v} \|_{L_m^2(\statespace,\R^d)}= \| v \|_{L_m^2(\statespace,\R^d)} \leq C$, by \eqref{eq:estim_bar_v}.
Using the Cauchy-Schwarz inequality and assumption \eqref{ass_holder}, we obtain that $|\smint \phi v m| \leq C$. We obtain then with assumption \eqref{ass_Psi_lin_growth} that for $P= \Psi( \smint \phi v m)$, we have $|P| \leq C$. Using assumption \eqref{ass_holder} and the continuity of $H_p$, we finally obtain that $\| v \|_{L^\infty(\statespace,\R^d)} \leq C$. Thus the bound \eqref{eq:bound_v_P} is satisfied.

\emph{Step 3:} uniqueness of $\mathbf{v}(t,m,w)$. \\
Let $v_1$ and $v_2 \in L^\infty(\statespace,\R^d)$ satisfy \eqref{eq:v_equal_H_p}. Then $DJ(v_i;t,m,w)= 0$,
proving that $v_1$ and $v_2$ are minimizers of $J(\cdot;t,m,w)$ and thus are equal in $L_m^2(\statespace,\R^d)$. Therefore $\smint \phi(x',t) v_1(x') m(x') \; \mathrm{d} x' = \smint \phi(x',t) v_2(x') m(x') \; \mathrm{d} x'$ and finally that $v_1= v_2$, by \eqref{eq:v_equal_H_p}.
\qed \end{proof}

\section{Regularity results for the auxiliary mappings} \label{section:auxiliary}

We provide in this section some regularity results for the mappings $\mathbf{v}$ and $\mathbf{P}$.
We begin by proving that $\mathbf{P}(\cdot,\cdot,\cdot)$ is locally H\"older continuous. For this purpose, we perform a stability analysis of the optimality condition \eqref{eq:opti_cond_v}.

\begin{lemma} \label{lemma:stability}
Let $t_1$ and $t_2 \in [0,T]$, let $w_1$ and $w_2 \in L^\infty(\statespace,\R^d)$, let $m_1$ and $m_2 \in \densities$. Let $R>0$ be such that $\| w_i \|_{L^\infty(\statespace,\R^d)} \leq R$,
for $i=1,2$.
Then, there exist a constant $C>0$ and an exponent $\alpha \in (0,1)$, both independent of $t_1$, $t_2$, $w_1$, $w_2$, $m_1$, and $m_2$ but depending on $R$, such that
\begin{align}
& | \mathbf{P}(t_2,m_2,w_2)-\mathbf{P}(t_1,m_1,w_1) | \notag \\
& \qquad \leq C \big( |t_2-t_1|^{\alpha} + \| w_2 - w_1 \|_{L^\infty(\statespace,\R^d)}^{\alpha} + \| m_2 - m_1 \|_{L^1(\statespace)}^{\alpha} \big). \label{eq:stability_2}
\end{align}
\end{lemma}

\begin{proof}
Note that all constants $C>0$ and all exponents $\alpha \in (0,1)$
involved below are independent of $t_1$, $t_2$, $w_1$, $w_2$, $m_1$,
and $m_2$. They are also independent of $x \in \statespace$ and
$\varepsilon > 0$.
For $i=1,2$, we set $v_i= \mathbf{v}(t_i,m_i,w_i)$ and $\phi_i= \phi(\cdot,t_i) \in L^\infty(\statespace)$. By \eqref{eq:bound_v_P}, we have
\begin{equation} \label{eq:bound_v_P_2}
\| v_i \|_{L^\infty(\statespace,\R^d)} \leq  C.
\end{equation}
By the optimality condition \eqref{eq:opti_cond_v}, we have that
\begin{equation} \label{eq:stat_v}
\phi_i^\intercal \Psi \big( t_i, \smint \phi_i v_i m_i \big)
+ L_v (t_i,v_i) + w_i= 0, \quad \text{for a.e. $x \in \statespace$}.
\end{equation}
Consider the difference of \eqref{eq:stat_v} for $i=2$ with \eqref{eq:stat_v} for $i=1$. Integrating with respect to $x$ the scalar product of the obtained difference with $v_2m_2 - v_1m_1$, we obtain that
$(a_1) + (a_2) + (a_3)= 0$,
where
\begin{align*}
(a_1)= \ & \int_{\statespace} \big\langle \phi_2^\intercal \Psi(t_2, \smint \phi_2 v_2 m_2) - \phi_1^\intercal \Psi(t_1, \smint \phi_1 v_1 m_1)  ,v_2 m_2 - v_1 m_1 \big\rangle \; \mathrm{d} x, \\
(a_2)= \ & \int_{\statespace} \langle L_v(t_2,v_2) - L_v(t_1,v_1), v_2 m_2 - v_1 m_1 \rangle \; \mathrm{d} x, \\
(a_3)= \ & \int_{\statespace} \langle w_2 - w_1, v_2 m_2 - v_1 m_1 \rangle \; \mathrm{d} x.
\end{align*}
We look for a lower estimate of these three terms.
Let us mention that the term $v_2 m_2 - v_1 m_1$, appearing in the three terms, will be estimated only at the end.

\emph{Estimation from below of $(a_1)$.} We have $(a_1)= (a_{11}) + (a_{12})$, where
\begin{align*}
(a_{11})= \ & \int_{\statespace} \big\langle \phi_2^\intercal \Psi(t_2,\smint \phi_2 v_2 m_2) - \phi_1^\intercal \Psi(t_1, \smint \phi_1 v_2 m_2), v_2 m_2 - v_1 m_1 \rangle \; \mathrm{d} x \\
(a_{12})= \ & \int_{\statespace} \big\langle \phi_1^\intercal \Psi(t_1, \smint \phi_1 v_2 m_2) - \phi_1^\intercal \Psi(t_1, \smint \phi_1 v_1 m_1 ), v_2 m_2 - v_1 m_1 \big\rangle \; \mathrm{d} x.
\end{align*}
By monotonicity of $\Psi$, we have that
\begin{equation*}
(a_{12})=  \Big\langle  \Psi(t_1, \smint \phi_1 v_2 m_2) - \Psi(t_1,
\smint \phi_1 v_1 m_1 ), \int_{\statespace} \phi_1 v_2 m_2 - \phi_1
v_1 m_1  \; \mathrm{d} x \Big\rangle \geq 0.
\end{equation*}
Let us consider $(a_{11})$. We set
\begin{equation*}
\begin{cases}
\begin{array}{rl}
\Psi_i= & \Psi(t_i, \smint \phi_i v_2 m_2), \quad \text{for $i=1,2$}, \\
\xi(x)= & \phi_2(x)^\intercal \Psi_2  - \phi_1(x)^\intercal \Psi_1,
\end{array}
\end{cases}
\end{equation*}
so that $(a_{11})= \int_{\statespace} \langle \xi, v_2 m_2 - v_1 m_1 \rangle \, \mathrm{d} x$.
Using assumption \eqref{ass_holder}, one can check that $|\Psi_i | \leq C$ and that $\big| \Psi_2 - \Psi_1 \big| \leq C|t_2- t_1|^\alpha$.
Since $\xi= (\phi_2-\phi_1)^\intercal \Psi_2 + \phi_1^\intercal (\Psi_2-\Psi_1)$, we obtain with assumption \eqref{ass_holder} again that
\begin{align*}
\| \xi \|_{L^\infty(\statespace,\R^d)} \leq C \big( |\Psi_2-\Psi_1| + \| \phi_2- \phi_1 \|_{L^\infty(\statespace,\R^{k \times d})} \big) \leq C |t_2-t_1|^\alpha
\end{align*}
and further with Young's inequality that
\begin{equation*}
|(a_{11})|
\leq \frac{C}{\varepsilon} |t_2-t_1|^\alpha + \frac{\varepsilon}{2} \| v_2 m_2 - v_1 m_1 \|_{L^1(\statespace,\R^d)}^2.
\end{equation*}

\emph{Estimation from below of $(a_2)$.} We have $(a_2)= (a_{21}) + (a_{22}) + (a_{23})$, where
\begin{align*}
(a_{21})= \ & \int_{\statespace} \langle L_v(t_2,v_2)- L_v(t_1,v_2), v_2 m_2 - v_1 m_1 \rangle \; \mathrm{d} x \\
(a_{22})= \ & \int_{\statespace} \langle L_v(t_1,v_2)- L_v(t_1,v_1), v_2(m_2-m_1) \rangle \; \mathrm{d} x \\
(a_{23}) = \ & \int_{\statespace} \langle L_v(t_1,v_2)- L_v(t_1,v_1), (v_2-v_1)m_1 \rangle \; \mathrm{d} x.
\end{align*}
As a consequence of \eqref{eq:bound_v_P_2}, assumption \eqref{ass_holder}, and Young's inequality, we have
\begin{align*}
|(a_{21})|
\leq \ & \frac{1}{2 \varepsilon} \| L_v(t_2,v_2(\cdot))-L_v(t_1,v_2(\cdot)) \|_{L^\infty(\statespace)}^2 + \frac{\varepsilon}{2} \| v_2 m_2 - v_1 m_1 \|_{L^1(\statespace,\R^d)}^2 \\
\leq  \ & \frac{C}{\varepsilon} |t_2 -t_1 |^\alpha + \frac{\varepsilon}{2} \| v_2 m_2 - v_1 m_1 \|_{L^1(\statespace,\R^d)}^2.
\end{align*}
By \eqref{eq:bound_v_P_2} and assumption \eqref{ass_holder}, $| L_v(t_1,x,v_i(x)) | \leq C$, therefore
\begin{equation*}
|(a_{22})| \leq C \| m_2 - m_1 \|_{L^1(\statespace,\R^d)}.
\end{equation*}
Finally, since $m_1 \geq 0$ and by assumption \eqref{eq:grad_monotony}, we have
\begin{equation*}
(a_{23}) \geq \frac{1}{C} \int_{\statespace} |v_2-v_1|^2 m_1 \; \mathrm{d} x.
\end{equation*}

\emph{Estimation from below of $(a_3)$.}
Using \eqref{eq:estim_v2m2-v1m1} and Young's inequality, we obtain that
\begin{equation*}
|(a_{3})|
\leq \frac{1}{2\varepsilon} \| w_2-w_1 \|_{L^\infty(\statespace,\R^d)}^2 + \frac{\varepsilon}{2} \| v_2 m_2 - v_1 m_1 \|_{L^1(\statespace,\R^d)}^2.
\end{equation*}

\noindent \emph{Conclusion.}
We have proved that
\begin{align}
& \frac{1}{C} \int_{\statespace} |v_2 -v_1|^2 m_1 \; \mathrm{d} x
\leq  (a_{23})  = (a_2) - (a_{21}) - (a_{22}) \notag \\
& \qquad \qquad = -(a_{1}) - (a_{21}) - (a_{22}) - (a_3) \notag \\
& \qquad \qquad \leq -(a_{11}) - (a_{21}) - (a_{22}) - (a_3) \notag \\
& \qquad \qquad \leq \frac{C}{\varepsilon} |t_2 -t_1 |^\alpha 
+ \frac{1}{2\varepsilon} \| w_2 - w_1 \|_{L^\infty(\statespace,\R^d)}^2
 \notag \\
& \qquad \qquad \qquad  + C \| m_2- m_1 \|_{L^1(\statespace,\R^d)} + \frac{3}{2} \varepsilon \| v_2 m_2 - v_1 m_1
                                                       \|_{L^1(\statespace;\R^d)}^2.
                                                       \label{eq:intermed_ineq}
\end{align}
Let us estimate $\| v_2 m_2 - v_1 m_1 \|_{L^1(\statespace;\R^d)}$. Using the Cauchy-Schwarz inequality, we obtain that
\begin{align}
& \| v_2 m_2 - v_1 m_1 \|_{L^1(\statespace,\R^d)}
\leq \| v_2(m_2-m_1) \|_{L^1(\statespace,\R^d)} + \| (v_2-v_1) m_1 \|_{L^1(\statespace,\R^d)} \notag \\
& \qquad \leq C \| m_2 - m_1 \|_{L^1(\statespace,\R^d)} + \Big( \int_{\statespace} |v_2- v_1 |^2 m_1 \; \mathrm{d} x \Big)^{1/2}. \label{eq:estim_v2m2-v1m1}
\end{align}
Injecting this inequality in \eqref{eq:intermed_ineq} 
and taking $\varepsilon= \frac{1}{3C}$, we obtain that
\begin{equation} \label{eq:stability_1}
\int_{\statespace} |v_2 -v_1|^2 m_1
\leq C \Big( |t_2 - t_1|^\alpha + \| m_2 - m_1 \|_{L^1(\statespace)} + \| w_2 - w_1 \|_{L^\infty(\statespace,\R^d)}^2 \Big).
\end{equation}
Let us prove \eqref{eq:stability_2}. We have
\begin{align*}
& \int_{\statespace} \phi_2 v_2 m_2 \; \mathrm{d} x - \int_{\statespace} \phi_1 v_1 m_1 \; \mathrm{d} x
= 
\int_{\statespace} (\phi_2-\phi_1) v_2 m_2 \; \mathrm{d} x \\
& \qquad \qquad
+ \int_{\statespace} \phi_1 v_2 (m_2-m_1) \; \mathrm{d} x
+ \int_{\statespace} \phi_1 (v_2-v_1) m_1 \; \mathrm{d} x.
\end{align*}
Therefore, using assumption \eqref{ass_holder} and \eqref{eq:stability_1}, we obtain that
\begin{align*}
& \Big| \int_{\statespace} \phi_2 v_2 m_2 \; \mathrm{d} x - \int_{\statespace} \phi_1 v_1 m_1 \; \mathrm{d} x \Big| \\
& \quad \leq C \Big( \| \phi_2-\phi_1 \|_{L^\infty(\statespace,\R^{k\times d})}
+ \| m_2 - m_1 \|_{L^1(\statespace)}
+ \Big( \int_{\statespace} |v_2-v_1|^2 m_1 \Big)^{1/2}
\Big) \\
& \quad \leq C \Big( |t_2-t_1|^{\alpha} + \| m_2 - m_1 \|_{L^1(\statespace)}^{1/2} + \| w_2 - w_1 \|_{L^\infty(\statespace,\R^d)} \Big).
\end{align*}
Inequality \eqref{eq:stability_2} follows, using assumption \eqref{ass_holder}. The lemma is proved.
\qed \end{proof}

Given $m \in L^\infty(0,T;\densities)$ and $w \in L^\infty(Q)$, we consider the Nemytskii operators associated with $\mathbf{v}$ and $\mathbf{P}$, that we still denote by $\mathbf{v}$ and $\mathbf{P}$ without risk of confusion:
\begin{align*}
& \mathbf{v}(m,w) \in L^\infty(Q,\R^d), && 
\mathbf{v}(m,w)(x,t) = \mathbf{v}(t,m(\cdot,t),w(\cdot,t))(x), \\
& \mathbf{P}(m,w) \in L^\infty(0,T;\R^k), &&
\mathbf{P}(m,w)(t)
= \mathbf{P}(t,m(\cdot,t),w(\cdot,t)),
\end{align*}
for all $(x,t) \in Q$.
We use now Lemma \ref{lemma:stability} to prove regularity properties
of the Nemytskii operators $\mathbf{v}$ and $\mathbf{P}$.
We recall that $X = \big( W^{2,1,p}(Q) \big)^2$.

\begin{lemma} \label{lemma:continuity_l_infty}
For all $R>0$, the mapping
\begin{align}
& (m,w) \in L^\infty(0,T;\densities) \times B \big( L^\infty(Q,\R^d),R) \notag \\
& \qquad \qquad \mapsto \mathbf{P}(m,w) \in L^\infty(0,T;\R^k) \label{eq:mapping_1}
\end{align}
and the mapping
\begin{align}
& (u,m) \in B(W^{2,1,p}(Q),R) \times L^\infty(0,T;\densities) \notag \\
& \qquad \qquad \mapsto \mathbf{v}(m,\nabla u) \in L^\infty(Q,\R^d) \cap L^p(0,T;W^{1,p}(\statespace))  \label{eq:mapping_2}
\end{align}
are both H\"older continuous, that is, there exist $\alpha \in (0,1)$ and $C>0$ such that
\begin{align*}
& \| \mathbf{P}(m_2,w_2)-\mathbf{P}(m_1,w_1) \|_{L^\infty(0,T;\R^k)} \\
& \qquad \qquad \leq C \big( \| m_2 - m_1 \|_{L^\infty(Q)}^\alpha + \| w_2 - w_1 \|_{L^\infty(Q)}^\alpha \big),\\[0.2em]
& \| \mathbf{v}(m_2,\nabla u_2)-\mathbf{v}(m_1,\nabla u_1) \|_{L^\infty(Q,\R^d) \cap L^p(0,T;W^{1,p}(\statespace))} \\
& \qquad \qquad \leq C \big( \| u_2 - u_1 \|_{W^{2,1,p}(Q)}^\alpha + \| m_2 - m_1 \|_{L^\infty(Q)}^\alpha \big),
\end{align*}
for all $m_1$ and $m_2 \in L^\infty(0,T;\densities)$, for all $w_1$ and $w_2 \in B(L^\infty(Q,\R^d),R)$, and for all $u_1$ and $u_2$ in $B(W^{2,1,p}(Q),R)$.
\end{lemma}

\begin{proof}
The H\"older continuity of the first mapping is a direct consequence of Lemma \ref{lemma:stability}. As a consequence, the mapping
\begin{align*}
& (u,m) \in B(W^{2,1,p}(Q),R) \times L^\infty(0,T;\densities) \\
& \qquad \qquad \qquad \mapsto \nabla u + \phi^\intercal \mathbf{P}(m,\nabla u) \in L^\infty(Q,\R^d)
\end{align*}
is H\"older continuous. Using then the relations
\begin{equation} \label{eq:formula_v}
\begin{array}{rl}
\mathbf{v}(m,\nabla u) = & - H_p(\nabla u + \phi^\intercal \mathbf{P}(m,\nabla u)), \\
D_x \mathbf{v}(m,\nabla u) = & - H_{px}(\nabla u + \phi^\intercal \mathbf{P}(m,\nabla u)) \\
& \quad - H_{pp}(\nabla u + \phi^\intercal \mathbf{P}(m,\nabla u))(\nabla^2 u + D\phi^\intercal \mathbf{P}(m,\nabla u)),
\end{array}
\end{equation}
and the H\"older continuity of $H_p$, $H_{px}$, and $H_{pp}$ on
bounded sets (Lemma \ref{lemma:reg_H}), we obtain that the second mapping is H\"older continuous.
\qed \end{proof}

\begin{remark} \label{remark:technical}
As a consequence of Lemma \ref{lemma:continuity_l_infty}, the images of the mappings given by \eqref{eq:mapping_1} and \eqref{eq:mapping_2} are bounded. This fact will be used in the steps 3 and 5 of the proof of Proposition \ref{prop:fixed_point_estimate}.
\end{remark}

\begin{lemma} \label{lemma:continuity_holder_1}
Let $R>0$ and $\beta \in (0,1)$. Then, there exists $\alpha \in (0,1)$ and $C>0$ such that for all $u \in B(W^{2,1,p}(Q),R)$ and for all $m \in B(\mathcal{C}^{\beta}(Q),R) \cap L^\infty(0,T;\densities)$,
$\| \mathbf{P}(m,\nabla u) \|_{\mathcal{C}^\alpha(0,T;\R^k)} \leq C$.
\end{lemma}

\begin{proof}
We recall that by Lemma \ref{lemma:max_reg_embedding}, $\| \nabla u \|_{\mathcal{C}^{\alpha}(Q,\R^d)} \leq C \| u \|_{W^{2,1,p}(Q)}$.
We obtain then the bound on $\| \mathbf{P}(m,\nabla u) \|_{\mathcal{C}^\alpha(0,T;\R^k)}$ with Lemma \ref{lemma:stability}.
\qed \end{proof}

\begin{lemma} \label{lemma:continuity_holder_2}
Let $R>0$ and $\beta \in (0,1)$. There exist $\alpha \in (0,1)$ and $C>0$ such that for all $u \in B(\mathcal{C}^{2+\beta,1+\beta/2}(Q),R)$ and for all $m \in B(\mathcal{C}^\beta(Q),R) \cap L^\infty(0,T;\densities)$,
\begin{equation*}
\| \mathbf{v}(m,\nabla u) \|_{\mathcal{C}^{\alpha}(Q,\R^d)} \leq C \quad \text{and} \quad
\| D_x \mathbf{v}(m,\nabla u) \|_{\mathcal{C}^{\alpha}(Q,\R^{d \times d})} \leq C.
\end{equation*} 
\end{lemma}

\begin{proof}
The result follows from relations \eqref{eq:formula_v}, Lemma \ref{lemma:continuity_holder_1}, and from the H\"older continuity of $H_p$, $H_{px}$, and $H_{pp}$ on bounded sets.
\qed \end{proof}

\section{A priori estimates for fixed points} \label{section:apriori}

\begin{proposition} \label{prop:fixed_point_estimate}
There exist a constant $C>0$ and an exponent $\alpha \in (0,1)$ such that for all $\tau \in [0,1]$, for all $(u_\tau,m_\tau,v_\tau,P_\tau) \in X'$ satisfying \eqref{MFGC_tau},
\begin{align*}
m_\tau \in & \ \mathcal{C}^{2+\alpha,1+ \alpha/2}(Q), && \| m_\tau \|_{\mathcal{C}^{2+\alpha,1+ \alpha/2}(Q)} \leq C, \\
u_\tau \in & \ \mathcal{C}^{2+\alpha,1+ \alpha/2}(Q), && \| u_\tau \|_{\mathcal{C}^{2+\alpha,1+ \alpha/2}(Q)} \leq C, \\
P_\tau \in & \ \mathcal{C}^\alpha(0,T;\R^k), && \| P_\tau \|_{\mathcal{C}^\alpha(0,T;\R^k)} \leq C, \\
v_\tau \in & \ \mathcal{C}^{\alpha}(Q,\R^d), && \| v_\tau \|_{\mathcal{C}^{\alpha}(Q,\R^d)} \leq C, \\
D_x v_\tau \in & \ \mathcal{C}^{\alpha}(Q,\R^{d\times d}), && \| D_x v_\tau \|_{\mathcal{C}^{\alpha}(Q,\R^{d \times d})} \leq C.
\end{align*}
\end{proposition}

\begin{proof}
Let us fix $\tau \in [0,1]$ and
$(u_\tau,m_\tau,v_\tau,P_\tau) \in X'$ satisfying
\eqref{MFGC_tau}. All constants $C$ and all exponents $\alpha \in
(0,1)$ involved below are independent of
$(u_\tau,m_\tau,v_\tau,P_\tau)$ and $\tau$. Let us recall that
$\tilde{f}_\tau \in L^\infty(Q)$ has been defined in Lemma
\ref{lemma:var_princ_1} by $\tilde{f}_\tau(x,t)= f(x,t,m_\tau(t))$.

\emph{Step 1:} $\| P_\tau \|_{L^2(0,T;\R^k)} \leq C$. \\
Let $v^0=0$ and let $m^0$ be the solution to
$\partial_t m^0 - \sigma \Delta m^0= 0$, $m^0(x,0) = m_0 (x)$.
By Lemma \ref{lemma:var_princ_1}, $B(m_\tau,v_\tau;\tilde{f}_\tau) \leq B(m^0,v^0;\tilde{f}_\tau)$.
Since $\| \phi \|_{L^\infty(Q,\R^{k \times d})} \leq C$, we have for all $\varepsilon >0$ and for all $t \in [0,T]$ that
\begin{align*}
& \Big| \int_{\statespace} \phi v_\tau m_\tau \; \mathrm{d} x \Big|
\leq C \int_{\statespace} |v_\tau| m_\tau \; \mathrm{d} x \\
& \qquad \leq C \Big( \int_{\statespace} |v_\tau|^2 m_\tau \; \mathrm{d} x \Big)^{1/2} 
\leq \frac{C}{\varepsilon} + C \varepsilon \int_{\statespace} |v_\tau|^2 m_\tau \; \mathrm{d} x,
\end{align*}
by the Cauchy-Schwarz inequality and Young's inequality.
The constant $C$ is also independent of $\varepsilon$.
Using then the lower bounds \eqref{L_quad_growth1} and
\eqref{eq:phi_linear_growth} and
assumptions \eqref{ass_f_bounded} and \eqref{ass_init_cond}, we obtain that
\begin{align*}
C \geq \ & B(m^0,v^0;\tilde{f}_\tau) \geq B(m_\tau,v_\tau;\tilde{f}_\tau) \\
\geq \ & \iint_Q \frac{1}{C} | v_\tau |^2 m_\tau \; \mathrm{d} x \; \mathrm{d} t
- C \Big| \int_{Q} \phi v_\tau m_\tau \; \mathrm{d} x \; \mathrm{d} t \Big| - C \\
\geq \ & \Big( \frac{1}{C} - C \varepsilon \Big) \iint_Q |v_\tau|^2 m_\tau \; \mathrm{d} x \; \mathrm{d} t - C\Big(1 + \frac{1}{\varepsilon} \Big).
\end{align*}
Taking $\varepsilon = 1 / (2C^2)$, we deduce that $\iint_Q |v_\tau|^2 m_\tau \, \mathrm{d} x \, \mathrm{d} t \leq C$.
Using then assumption \eqref{ass_Psi_lin_growth}, the boundedness of
$\phi$, the Cauchy-Schwarz inequality and the estimate obtained
previously, we deduce that
\begin{align}
\| P_\tau \|_{L^2(0,T;\R^k)}
= \ & \int_0^T | \Psi(t, \smint \phi v_\tau m_\tau) |^2 \; \mathrm{d} t
\leq C+C \int_0^T \Big| \int_{\statespace} \phi v_\tau m_\tau \; \mathrm{d} x \Big|^2 \mathrm{d} t \notag \\
\leq \ & C + C\iint_Q |v_\tau|^2 m_\tau \; \mathrm{d} x \; \mathrm{d} t \leq C. \label{eq:p_tau_l2}
\end{align}

\emph{Step 2:} $\| u_\tau \|_{L^\infty(Q)} \leq C$, $\| \nabla u_\tau \|_{L^\infty(Q,\R^d)} \leq C$.\\
The argument is classical. We have that $u_\tau$ is the unique solution to the HJB equation $(i)$\eqref{MFGC_tau}. It is therefore the value function associated with the following stochastic optimal control problem:
\begin{equation}\label{cost2}
u_\tau(x,t) = \tau \Big( \, \inf_{\alpha \in L_{\mathbb{F}}^2(t,T;\R^d)} J_\tau(x,t,\alpha) \Big),
\end{equation}
where $J_\tau(x,t,\alpha)$ is defined by
\begin{align*}
\E \Big[ \int_{t}^{T} \big( L(X_s,s, \alpha_s) +
\myprod{\phi(X_s,s)^\intercal P_\tau (s)}{\alpha_s} +  \tilde{f}_\tau(X_s,s) \big) \; \mathrm{d} s + g(X_T)  \Big],
\end{align*}
and
$(X_s)_{s \in [t,T]}$ is the solution to the stochastic dynamic $\mathrm{d} X_s = \tau \alpha_s \mathrm{d} s + \sqrt{2\sigma} \mathrm{d} B_s, \; X_t = x$.
Here, $L_{\mathbb{F}}^2(t,T;\R^d)$ denotes the set of stochastic processes on $(t,T)$, with values in $\R^d$, adapted to the filtration $\mathbb{F}$ generated by the Brownian motion $(B_s)_{s \in [0,T]}$, and such that
$\mathbb{E}\big[ \int_t^T |\alpha(s)|^2 \; \mathrm{d} s \big] < \infty$.
Then, the boundedness of $u_\tau$ from above can be immediately
obtained by choosing $\alpha = 0$ in \eqref{cost2} and using the
boundedness of $g$. We can as well bound $u_\tau$ from below since for
all $(x,s) \in Q$ and for all $\alpha \in \R^d$, we have
\begin{align*}
L(x,s,\alpha) + \myprod{\phi(x,s)^\intercal P_\tau (s)}{\alpha}
\geq \ &
\frac{1}{C} | \alpha |^2 - \| \phi \|_{L^\infty(Q,\R^{k \times d})}  | P_\tau (s) | | \alpha | -C \\
\geq \ & \frac{1}{C} | \alpha |^2 - C |P_\tau(s)|^2 -C,
\end{align*}
for some constant $C$ independent of $(x,s)$, $\alpha$, and $P_\tau(s)$.
We already know from the previous step that $ \| P_\tau \|_{L^2(0,T;\R^k)} \leq C$. So we can conclude that $u_\tau$ is also bounded from below, and thus $\| u_\tau \|_{L^\infty(Q)} \leq C$. We also deduce from the above inequality that for all $\alpha \in L_{\mathbb{F}}^2(t,T;\R^d)$,
\begin{equation} \label{eq:bound_norm_alpha}
\mathbb{E} \Big[ \int_t^T |\alpha_s|^2 \; \mathrm{d} s \Big] \leq C \big( J_\tau(x,t,\alpha) + 1 \big).
\end{equation}
 
Let us bound $\nabla u_\tau$. Choose $\varepsilon \in (0,1)$.
For  arbitrary $(x,t)$, take an $\varepsilon$-optimal stochastic optimal control $\tilde \alpha$ for \eqref{cost2}. We can deduce from the boundedness of the map $u_\tau$ and inequality \eqref{eq:bound_norm_alpha} that
\begin{equation} \label{eq:bound_alpha_tilde}
\mathbb{E} \Big[ \int_t^T |\tilde{\alpha}_s|^2 \; \mathrm{d} s \Big]
\leq C \big( J_\tau(x,t,\alpha) + 1 \big)
\leq C(u_\tau(x,t)+ \varepsilon + 1)
\leq C,
\end{equation}
where $C$ is independent of $(\tau,x,t)$ and $\varepsilon$.
Let $y \in \statespace$. Set
\begin{equation}
\mathrm{d} X_s = \tau \tilde \alpha_s \mathrm{d} s + \sqrt{2\sigma} \mathrm{d} B_s , \; X_t = x, \quad \text{and} \quad Y_s = X_s - x + y,
\end{equation}
then obviously $\mathrm{d} Y_s = \tilde \alpha_s \mathrm{d} s + \sqrt{2\sigma} \mathrm{d} B_s , \; Y_t = y$.
We have
\begin{align*}
u_\tau(x,t) + \varepsilon \geq &
\ \tau \E \Big[ \int_{t}^{T} L(X_s,s, \tilde\alpha_s) + \myprod{P_\tau (s)}{ \phi(X_s,s)^\intercal \tilde \alpha_s} \; \mathrm{d} s \\
& \qquad \qquad + \int_t^T \tilde{f}_\tau(X_s,s) \; \mathrm{d} s + g(X_T)  \Big],
\end{align*}
\begin{align*}
u_\tau(y,t) \leq & \ \tau \E \Big[ \int_{t}^{T} \Big( L(Y_s,s,
                   \tilde\alpha_s) +  \myprod{ \phi(Y_s , s)^\intercal
                   P_\tau (s)}{ \tilde \alpha_s} \; \mathrm{d} s \\
& \qquad \qquad  + \int_t^T \tilde{f}_\tau(Y_s,s) \Big)  \; \mathrm{d} s + g(Y_T)  \Big].
\end{align*}
Therefore,
$u_\tau(y,t) - u_\tau(x,t) \leq \varepsilon + |(a)| + |(b)| + |(c)|  + |(d)|$,
where $(a)$, $(b)$, $(c)$, $(d)$ are given by
\begin{align*}
(a)= \ & \tau \mathbb{E} \Big[ \int_t^T L(Y_s,s,\tilde{\alpha}_s) - L(X_s,s,\tilde{\alpha}_s) \; \mathrm{d} s \Big], \\
(b)= \ & \tau \mathbb{E} \Big[ \int_t^T \big( \phi(Y_s,s) - \phi(X_s,s) \big)^\intercal P_\tau(s), \tilde{\alpha}_s \rangle \; \mathrm{d} s \Big] , \\
(c)= \ & \tau \mathbb{E} \big[ g(Y_T)-g(X_T) \big], \\
  (d)= \ & \tau \mathbb{E} \Big[ \int_t^T \big(
           \tilde{f}_\tau(Y_s,s)  - \tilde{f}(X_s,s)  \big) \; \mathrm{d} s \Big].
\end{align*}
First, we have
\begin{align*}
|(a)| \leq \ & \tau
\mathbb{E} \Big[ \int_t^T \big| L(Y_s,s, \tilde\alpha_s) - L(X_s,s,\tilde\alpha_s) \big| \; \mathrm{d} s \Big] \\
\leq \ & C |y-x| \Big( 1 + \mathbb{E} \Big[ \int_t^T | \tilde{\alpha}_s |^2 \; \mathrm{d} s \Big] \Big)
\leq C |y-x|,
\end{align*}
as a consequence of assumption \eqref{ass_L_Lipschitz} and \eqref{eq:bound_alpha_tilde}.
Then, using assumption \eqref{ass_holder}, \eqref{eq:p_tau_l2}, and \eqref{eq:bound_alpha_tilde}, we obtain
\begin{align*}
|(b)| \leq \ & \tau \mathbb{E} \Big[ \int_t^T |\phi(Y_s , s)-\phi(X_s , s)| |P_\tau(s)| |\tilde{\alpha}(s)| \; \mathrm{d} s \Big] \\
\leq \ & C |y-x| \, \| P_\tau \|_{L^2(0,T;\R^k)} \, \mathbb{E} \Big[ \int_t^T |\tilde{\alpha}(s)|^2 \; \mathrm{d} s \Big]
\leq C |y-x|.
\end{align*}
By assumption \eqref{ass_init_cond},
$|(c)| \leq \mathbb{E} \big[ |g(Y_T)-g(X_T)| \big] \leq C |y-x|$.
Finally, since $\tilde{f}_\tau$ is a Lipschitz function (by assumption \eqref{ass_hold_f}),
\begin{align}
|(d)| \leq \ & \tau \mathbb{E} \Big[ \int_t^T \big| \tilde{f}_\tau(Y_s,s)  - \tilde{f}_\tau(X_s,s)  \big| \; \mathrm{d} s \Big]
               \leq C |y-x|.
\end{align}
Letting $\varepsilon \to 0$, we obtain that $u_\tau(y,t) - u_\tau(x,t) \leq C |y-x|$. Exchanging $x$ and $y$, we obtain that $u_\tau$ is Lipschitz continuous with modulus $C$ and finally that $\| \nabla u_\tau \|_{L^\infty(Q,\R^d)} \leq C$.

\emph{Step 3:} $\| P_\tau \|_{L^\infty(0,T;\R^k)} \leq C$.\\
By Lemma \ref{lemma:max_principle}, $m_\tau \in L^\infty(0,T;\densities)$.
We have that $\| \nabla u_\tau \|_{L^\infty(Q,\R^d)} \leq C$ and $P_\tau= \mathbf{P}(m_\tau,\nabla u_\tau)$.
The bound on $\| P_\tau \|_{L^\infty(0,T;\R^k)}$ follows then from Lemma \ref{lemma:continuity_l_infty} and Remark \ref{remark:technical}.

\emph{Step 4:} $\| u_\tau \|_{W^{2,1,p}(Q)} \leq C$.\\
By assumption \eqref{ass_holder}, $\phi$ is bounded. We have proved
that $\| P_\tau \|_{L^\infty(0,T;\R^k)} \leq C$ and by Lemma
\ref{lemma:reg_H}, $H$ is continuous. Thus, $\|
H(\nabla u_\tau + \phi^\intercal P_\tau) \|_{L^\infty(Q)} \leq C$.
By assumption \eqref{ass_f_bounded}, $\| \tau \tilde{f}_\tau \|_{L^\infty(Q)} \leq C$.
It follows that $u_\tau$, as the solution to the HJB equation $(i)$\eqref{MFGC_tau}, is the solution to a parabolic equation with bounded coefficients. Thus, by Theorem \ref{theo:max_reg2}, $\| u_\tau \|_{W^{2,1,p}(Q)} \leq C$. 
We also obtain with Lemma \ref{lemma:max_reg_embedding} that $\| u_\tau \|_{\mathcal{C}^\alpha(Q)} \leq C$ and $\| \nabla u_\tau \|_{\mathcal{C}^\alpha(Q,\R^d)} \leq C$.

\emph{Step 5:} $\| v_\tau \|_{L^\infty(Q,\R^d)} \leq C$, $\| D_x v_\tau \|_{L^p(Q,\R^{d \times d})} \leq C$. \\
We have proved that $v_\tau = \mathbf{v}(m_\tau,\nabla u_\tau)$ and $\| u_\tau \|_{W^{2,1,p}(Q)} \leq C$. The estimate follows directly with Lemma \ref{lemma:continuity_l_infty} and Remark \ref{remark:technical}. 

\emph{Step 6:} $\| m_\tau \|_{\mathcal{C}^\alpha(Q)} \leq C$. \\
The Fokker-Planck equation can be written in the form of a parabolic equation with coefficients in $L^p$:
$\partial_t m_{\tau} - \sigma \Delta m_\tau + \tau \langle v_\tau, \nabla m_\tau \rangle + \tau m_\tau \text{div}(v_\tau)= 0$,
since $\| D_x v_\tau \|_{L^p(Q,\R^{d \times d})} \leq C$.
Combining Theorem \ref{theo:max_reg1}
and Lemma \ref{lemma:max_reg_embedding},
we get that $\| m_\tau \|_{\mathcal{C}^\alpha(Q)} \leq C$.

\emph{Step 7:} $\| P_\tau \|_{\mathcal{C}^\alpha(0,T;\R^k)} \leq C$. \\
We already know that $\| u_\tau \|_{W^{2,1,p}(Q)} \leq C$, that $\| m_\tau \|_{\mathcal{C}^\alpha(Q)} \leq C$, and that $m_\tau \in L^\infty(0,T;\densities)$. Thus Lemma \ref{lemma:continuity_holder_1} applies and yields that $\| P_\tau \|_{\mathcal{C}^\alpha(0,T;\R^k)} \leq C$.

\emph{Step 8:} $\| u_\tau \|_{\mathcal{C}^{2+\alpha,1+ \alpha/2}(Q)} \leq C$. \\
We have proved that $\| \nabla u_\tau \|_{\mathcal{C}^\alpha(Q,\R^d)}
\leq C$ and $\| P_\tau \|_{\mathcal{C}^\alpha(0,T;\R^k)} \leq
C$. Moreover, we have assumed that $\phi$ is H\"older continuous and
know that $H$ is H\"older continuous on bounded sets.
It follows that
$\| H(\nabla u_\tau + \phi^\intercal P_\tau) \|_{\mathcal{C}^\alpha(Q)} \leq C$.
It follows from assumption \eqref{ass_hold_f} that $\tau \tilde{f}_\tau$ is H\"older continuous.
Since $g \in \mathcal{C}^{2+\alpha}(\statespace)$, we finally obtain that $\| u_\tau \|_{\mathcal{C}^{2+\alpha,1+ \alpha/2}(Q)} \leq C$, by Theorem \ref{theo:holder_reg_classical}.

\emph{Step 9:} $\| v_\tau \|_{\mathcal{C}^\alpha(0,T;\R^{d})} \leq C$ and $\| D_x v_\tau \|_{\mathcal{C}^\alpha(0,T;\R^{d \times d})} \leq C$. \\
We have $\| u_\tau \|_{\mathcal{C}^{2+\alpha,1+\alpha/2}(Q)} \leq C$ and $\| m_\tau \|_{\mathcal{C}^\alpha(Q)} \leq C$. Thus Lemma \ref{lemma:continuity_holder_2} applies and the announced estimates hold true.

\emph{Step 10:} $\| m_\tau \|_{\mathcal{C}^{2+ \alpha, 1 + \alpha /2}(Q)} \leq C$. \\
A direct consequence of Step 9 is that $m_\tau$ is the solution to a parabolic equation with H\"older continuous coefficients. Therefore $\| m_\tau \|_{\mathcal{C}^{2+\alpha,1 + \alpha/2}(Q)} \leq C$, by Theorem \ref{theo:holder_reg_classical}, which concludes the proof of the proposition.
\qed \end{proof}

\section{Application of the Leray-Schauder theorem} \label{section:leray-schauder}

\begin{proof}[Theorem \ref{theo:main}]
\emph{Step 1:} construction of $\mathcal{T}$. \\
Let us define the mapping $\mathcal{T} \colon X \times [0,1] \rightarrow X$ which is used for the application of the Leray-Schauder theorem.
A difficulty is that the auxiliary mappings $\mathbf{P}$ and $\mathbf{v}$ are only defined for $m \in L^\infty(0,T;\densities)$. Therefore we need a kind of projection operator on this set.
Note that
$\int_{\statespace} 1 \, \mathrm{d} x = 1$. We consider the mapping
\begin{equation*}
\rho: m \in L^\infty(Q) \mapsto \rho(m) \in L^\infty(0,T;\densities),
\end{equation*}
defined by
\begin{equation*}
\rho(m)= \frac{m_+(x,t)}{\max ( 1, \smint m_+(y,t) \; \mathrm{d} y )}
+ 1 - \frac{\smint m_+(y,t) \; \mathrm{d} y}{\max ( 1, \smint m_+(y,t) \; \mathrm{d} y )},
\end{equation*}
where $m_+(x,t)= \max(0,m(x,t))$.
For checking that $\rho(m) \in L^\infty(0,T;\densities)$, we suggest to consider the two cases: $\smint m_+(y,t) \; \mathrm{d} y < 1$ and $\smint m_+(y,t) \; \mathrm{d} y \geq 1$ separetely.
The following properties can be easily checked:
\begin{itemize}
\item For all $m \in L^\infty(0,T;\densities)$, $\rho(m)=m$.
\item The mapping $\rho$ is locally Lipschitz continuous, from $L^\infty(Q)$ to $L^\infty(Q)$.
\item For all $\alpha \in (0,1)$, there exists a constant $C>0$ such that if $m \in \mathcal{C}^\alpha(Q)$, then $\rho(m) \in \mathcal{C}^\alpha(Q)$ and $\| \rho(m) \|_{\mathcal{C}^{\alpha}(Q)} \leq C \| m \|_{\mathcal{C}^{\alpha}(Q)}$.
\end{itemize}

For a given $(u,m,\tau) \in X \times [0,1]$, the pair $(\tilde{u},\tilde{m})= \mathcal{T}(u,m,\tau)$ is defined as follows: $\tilde{u}$ is the solution to
\begin{equation*}
\begin{cases}
\begin{array}{rll}
-\partial_t \tilde{u} -\sigma \Delta \tilde{u} + \tau H(\nabla u +
  \phi^\intercal \mathbf{P}(\rho(m),\nabla u) ) = & \tau f(\rho(m(t))) \quad &(x,t) \in Q, \\
\tilde{u}(x,T) = & \tau g(x) \quad & x \in \statespace,
\end{array}
\end{cases}
\end{equation*}
and $\tilde{m}$ is the solution to
\begin{equation*}
\begin{cases}
\begin{array}{rll}
\partial_t \tilde{m} - \sigma \Delta \tilde{m} + \tau \mathrm{div} (\mathbf{v}(\rho(m),\nabla \tilde{u}) m )= & 0 \quad &(x,t) \in Q, \\
\tilde{m}(x,0)= & m_0(x) & x \in \statespace.
\end{array}
\end{cases}
\end{equation*}
It directly follows from the definition of $\mathcal{T}$ that $\mathcal{T}(u,m,0)$ is constant, as required by the Leray-Schauder theorem.

\emph{Step 2:} a priori bound. \\
Let $\tau \in [0,1]$ and let $(u_\tau,m_\tau)$ be such that $(u_\tau,m_\tau)= \mathcal{T}(u_\tau,m_\tau,\tau)$. Then, by Lemma \ref{lemma:max_principle}, $m_\tau \in L^\infty(0,T;\densities)$. Thus, $m_\tau= \rho(m_\tau)$ and finally, by Lemma \ref{lemma:stat_existence}, the quadruplet $(u_\tau,m_\tau,P_\tau,v_\tau)$, with $P_\tau= \mathbf{P}(m_\tau,\nabla u_\tau)$ and $v_\tau= \mathbf{v}(m_\tau,\nabla u_\tau)$, is a solution to \eqref{MFGC_tau}. We directly conclude with Proposition \ref{prop:fixed_point_estimate} that $\| (u_\tau,m_\tau) \|_X \leq C$, where $C$ is independent of $\tau$.

\emph{Step 3:} continuity of $\mathcal{T}$. \\
Using the continuity of $\rho$, Lemma \ref{lemma:continuity_l_infty}, the H\"older continuity of $H$, and assumption \eqref{ass_hold_f}, we obtain that the mappings
\begin{align*}
& (u,m) \in X \mapsto H(\nabla u + \phi^\top \mathbf{P}(\rho(m),\nabla u)) - f(\rho(m)) \in L^\infty(Q), \\
& (u,m) \in X \mapsto \text{div}(\mathbf{v}(\rho(m),\nabla u)m) \in L^p(Q)
\end{align*}
are continuous.
By Theorem \ref{theo:max_reg2}, the solution to a parabolic equation of the form \eqref{eq:parabolic}, with $b$ and $c$ null (in $W^{2,1,p}(Q)$) is a continuous mapping of the right-hand side (in $L^p(Q)$).
Thus, $\tilde{u} \in W^{2,1,p}(Q)$ depends in a continuous way on $\tau H(\nabla u + \phi^\top \mathbf{P}(\rho(m),\nabla u))$ and therefore $\tilde{u}$ depends in a continuous way on $(\tau,u,m)$ by composition.
Again, by Theorem \ref{theo:max_reg2}, $\tilde{m} \in W^{2,1,p}(Q)$ depends in a continuous way on $\tau \text{div}(\mathbf{v}(\rho(m),\nabla \tilde{u})m)$ and therefore depends in a continuous way on $(\tau,u,m)$.

\emph{Step 4:} compactness of $\mathcal{T}$. \\
Let $R>0$, let $(u,m) \in B(X,R)$.
We have $\| \rho(m) \|_{\mathcal{C}^{\alpha}(Q)} \leq C$, where $C$ is independent of $(u,m)$ (but depends on $R$).
As a consequence of assumption \eqref{ass_hold_f}, and since $H$ is H\"older continuous on bounded sets, we have
\begin{align*}
& \| H(\nabla u + \phi^\intercal \mathbf{P}(\rho(m),\nabla u))- f(\rho(m)) \|_{\mathcal{C}^\alpha(Q)} \leq C, \\
\end{align*}
where $C>0$ and $\alpha \in (0,1)$ are both independent of $(u,m)$ (but depend on $R$). It follows then that $\| u \|_{\mathcal{C}^{2+\alpha,1+\alpha/2}(Q)} \leq C$ by Theorem \ref{theo:holder_reg_classical}.
Using Lemma \ref{lemma:continuity_holder_2}, we deduce then that
\begin{equation*}
\| \text{div}(\mathbf{v}(\rho(m),\nabla \tilde{u}) m) \|_{\mathcal{C}^\alpha(Q)} \leq C,
\end{equation*}
and finally obtain that $\| m \|_{\mathcal{C}^{2+\alpha,1+ \alpha/2}(Q)} \leq C$, by Theorem \ref{theo:holder_reg_classical} again.
The compactness of $\mathcal{T}$ follows, since $\mathcal{C}^{2+\alpha,1+\alpha/2}(Q)$ is compactly embedded in $W^{2,1,p}(Q)$, by the Arzel{\`a}-Ascoli theorem.

\emph{Step 5:} conclusion. \\
The existence of a fixed point $(u,m)$ to $\mathcal{T}(\cdot,\cdot,1)$ follows. With the same arguments as those of Step 2, we obtain that $(u,m,\mathbf{P}(m,\nabla u),\mathbf{v}(m,\nabla u))$ is a solution to \eqref{MFGC_tau} with $\tau=1$ and that \eqref{eq:main} holds, by Proposition \ref{prop:fixed_point_estimate}.
\qed \end{proof}

\section{Uniqueness and duality} \label{section:duality}

In this section we prove the uniqueness of the solution $(u,m,v,P)$ to \eqref{MFGC}. We also prove that $(P,v)$ is the solution to a dual problem to \eqref{eq:primal_problem}.
Both results are obtained under the following additional monotonicity assumption of $f$:
There exists a measurable mapping $F(t,m) \colon [0,T]\times \densities \to \R$
such that
\begin{equation}
\label{subdifF}
F(t,m_2) - F(t,m_1) \geq \int_{ \statespace} f(x,t,m_1) (m_2(x)-m_1(x)) \; \mathrm{d} x,
\end{equation}
for all $m_1$ and $m_2 \in \densities$ and for a.e.\@ $t$.
Thus, $F(t,\cdot)$ is a supremum of the exact affine minorants appearing in the above right-hand side, and is therefore a convex function of $m$.

\begin{remark} \label{rem:pot}
\begin{enumerate}
\item It follows from \eqref{subdifF} that $f$ is monotone:
\begin{equation} \label{eq:monotony_f}
\int_{ \statespace } (f(x,t,m_2)- f(x,t,m_2)) (m_2(x) - m_1(x)) \; \mathrm{d} x
\geq 0,
\end{equation}
for all $m_1$ and $m_2 \in \densities$ and for a.e.\@ $t$.
Conversely, \eqref{subdifF} holds true if \eqref{eq:monotony_f} is satisfied and if $F$ is a primitive of $f(.,t,.)$ in the sense that
\begin{equation*}
F(t,m_2) - F(t,m_1) = \int_{0}^{1} \int_{\statespace} f(x,t, sm_2 + (1-s)m_1) (m_2(x) - m_1(x)) \; \mathrm{d} s.
\end{equation*}
We refer to \cite[Proposition 1.2]{CDHD2015} for a further characterization of functions $f$ deriving from a potential.
\item Consider the mapping $f_K$ proposed in Lemma
  \ref{lemma:ex_f}. Assume that for all $(x,t) \in Q$, $K(x,t,\cdot)$
  is non-decreasing and consider the function $\mathcal{K}$ defined
  by
  $\mathcal{K}(x,t,w)  := \int_0^w K(x,t,w') \; \mathrm{d} w'$, for $(x,t,w) \in Q \times [-C_2,C_2]$. Then inequality \eqref{subdifF} holds true with $F_K$ defined by
\begin{equation*}
F_K(t,m)= \int_{\statespace} \mathcal{K} (x,t,m * \varphi(x)) \; \mathrm{d} x.
\end{equation*}
Indeed, since $\mathcal{K}$ is convex in its third argument, we have
\begin{align*}
& F_K(t,m_2)-F_K(t,m_1)
= \int_{\statespace} \mathcal{K}(x,t,m_2 * \varphi(x)) - \mathcal{K}(x,t,m_1 * \varphi(x)) \; \mathrm{d} x \\
& \qquad \qquad \geq \int_{\statespace} K(x,t,m_1* \varphi(x))((m_2-m_1)*\varphi)(x) \; \mathrm{d} x \\
& \qquad \qquad = \int_{\statespace} (K(\cdot,t,m_1*\varphi(\cdot))* \tilde{\varphi})(x) (m_2(x) - m_1(x)) \; \mathrm{d} x \\
& \qquad \qquad = \int_{\statespace} f_K(x,t,m)(m_2(x)-m_1(x)) \; \mathrm{d} x,
\end{align*}
as was to be proved.
\end{enumerate}
\end{remark}

Without loss of generality, we can assume that $F(t,m_0)=0$ for a.e.\@ $t \in (0,T)$. It can then be easily deduced from assumption \eqref{ass_f_bounded} and \eqref{subdifF} that there exists a constant $C$ such that
\begin{equation} \label{eq:F_bounded}
|F(t,m)| \leq C, \quad \forall m \in \densities, \text{ for a.e.\@ $t \in (0,T)$}.
\end{equation}
Let us consider the potential $B \colon W^{2,1,p}(Q) \times L^\infty(Q;\R^k) \rightarrow \R$, defined by
\begin{align}
& B (m,v) =
\iint_Q L ( x,t, v(x,t)) m(x,t) \; \mathrm{d} x \; \mathrm{d} t + \int_0^T F(t,m(t)) \; \mathrm{d} t \notag \\
& \qquad + \int_0^T \Phi \Big(t, \int_{\statespace} \phi(x,t) v(x,t)  m(x,t) \; \mathrm{d} x \Big) \, \mathrm{d} t
+ \int_{\statespace} g(x) m(x,T) \; \mathrm{d} x. \label{potential2}
\end{align}

\begin{proposition}
There exists a unique solution $(u,m,v,P) \in X'$ to \eqref{MFGC}. Moreover, the pair $(m,v)$ is the solution to the following optimal control problem
\begin{equation} \label{eq:primal_problem_2}
\min_{\begin{subarray}{c} \hat{m} \in W^{2,1,p}(Q) \\ \hat{v} \in
    L^\infty(Q,\R^d) \end{subarray}} \
B(\hat{m},\hat{v}), \quad
\text{s.t.: }
\begin{cases}
\begin{array}{rl}
\partial_t \hat{m} - \sigma \Delta \hat{m} + \mathrm{div} (\hat{v} \hat{m})= & 0, \\
\hat{m}(x,0)= & m_0(x).
\end{array}
\end{cases}
\end{equation}
\end{proposition}

\begin{proof}
Let $(u,m,v,P) \in X'$ be a solution to \eqref{MFGC}. Let us prove that $(m,v)$ is a solution to \eqref{eq:primal_problem_2}. Let $(\hat{m},\hat{v})$ be a feasible pair. Denoting $\tilde{f}(x,t)= f(x,t,m(t))$, we have 
\begin{align*}
& B(\hat{m},\hat{v})-B(m,v)
= \big( B(\hat{m},\hat{v};\tilde{f})- B(m,v;\tilde{f}) \big) \\
& \qquad + \Big( \int_0^T F(t,\hat{m}(t))- F(t,m(t)) - \int_{\statespace} \tilde{f}(x,t) (\hat{m}(x,t)-m(x,t)) \; \mathrm{d} x \; \mathrm{d} t \Big).
\end{align*}
The two terms in the right-hand side are both nonnegative, as a consequence of Lemma \ref{lemma:var_princ_1} and assumption \eqref{subdifF}, respectively.

It remains to prove the uniqueness of the solution to \eqref{MFGC}.
Let us prove first a classical property: There exists a constant $C>0$ such that for all $(x,t) \in Q$, for all $p \in \R^d$ and for all $v \in \R^d$,
\begin{equation} \label{eq:strong_duality}
H(x,t,p) + L(x,t,v) + \langle p, v \rangle \geq \frac{1}{2C} |v + H_p(x,t,p)|^2.
\end{equation}
Let us set $\bar{v}= - H_p(x,t,p)$. For a fixed triple $(x,t,p)$, we have $H(x,t,p)= - \langle p,\bar{v} \rangle - L(x,t,\bar{v})$. Moreover, $L_v(x,t,\bar{v})= -p$ and thus by \eqref{eq:grad_monotony},
\begin{equation*}
L(x,t,v) \geq L(x,t,\bar{v}) - \langle p, v- \bar{v} \rangle + \frac{1}{2C} |v-\bar{v}|^2.
\end{equation*}
Inequality \eqref{eq:strong_duality} follows.

Let $(u_1,m_1,v_1,P_1)$ and $(u_2,m_2,v_2,P_2)$ be two solutions to \eqref{MFGC} in $X'$. We obtain with inequality \eqref{eq:strong_duality} that
\begin{align*}
L (v_2) \geq \ & - H(\nabla u_1 + \phi^\intercal \price_1) - \myprod{\nabla u_1 + \phi^\intercal \price_1}{v_2} + \frac{1}{2C} |v_2-v_1|^2,  \\
L (v_1 ) = \ & - H(\nabla u_1 + \phi^\intercal \price_1 ) - \myprod{\nabla u_1 + \phi^\intercal \price_1 }{v_1}.
\end{align*}
Proceeding then exactly like in the proof of Lemma \ref{lemma:var_princ_1}, we arrive at the following inequality:
\begin{equation*}
B(m_2,v_2)-B(m_1,v_1) \geq \frac{1}{2C} \iint_Q |v_2 - v_1 |^2 m_2 \; \mathrm{d} x \; \mathrm{d} t.
\end{equation*}
We also have that $B(m_1,v_1)-B(m_2,v_2) \geq 0$, thus $\iint_Q |v_2 - v_1 |^2 m_2 \; \mathrm{d} x \; \mathrm{d} t= 0$. As a consequence, $(v_2-v_1)m_2= 0$, since $m_2 \geq 0$. We obtain then that
\begin{equation} \label{eq:div_term}
v_2m_2 - v_1m_1= v_1(m_2-m_1).
\end{equation}
Let us set $m= m_2- m_1$. Using relation \eqref{eq:div_term}, we obtain that $m$ is the solution to the following parabolic equation: $\partial_t m - \sigma \Delta m + \text{div}(v_1 m)= 0$, $m(x,0)=0$.
Therefore $m=0$ and $m_2= m_1$. We already know that $v_2m_2= v_1m_2$, we deduce then that $v_2m_2= v_1m_1$. We obtain further with $(iii)$ that $P_1= P_2$, then with $(i)$ that $u_1= u_2$ and finally with $(iv)$ that $v_1= v_2$, which concludes the proof.
\qed \end{proof}

We finish this section with a duality result. For $\gamma \in L^\infty(\statespace)$, we recall that the convex conjugate of $F(t,\cdot)$ is defined by
\begin{equation*}
F^*(t,\gamma)= \sup_{m \in \densities} \int_{\statespace} \gamma(x) m(x) \; \mathrm{d} x - F(t,m).
\end{equation*}
It directly follows from the above definition that $|F^*(t,\gamma)| \leq \| \gamma \|_{L^\infty(\statespace)} + C$, where $C$ is the constant obtained in \eqref{eq:F_bounded} and thus for $\gamma \in L^\infty(Q)$, the integral $\int_0^T F^*(t,\gamma(\cdot,t)) \; \mathrm{d} t$ is well-defined.

Consider the dual criterion $D \colon (u,P,\gamma) \in W^{2,1,p}(Q) \times L^\infty(0,T;\R^k) \times L^\infty(Q) \mapsto D(u,p,\gamma) \in \R \cup \{ - \infty \}$,
defined by
\begin{equation*}
D(u,P,\gamma)= \int_{\statespace} u(x,0) m_0(x) \; \mathrm{d} x - \int_0^T
\Phi^*(t,P(t)) \; \mathrm{d} t
- \int_0^T F^*(t,\gamma(t)) \; \mathrm{d} t.
\end{equation*}
The function $\Phi^*$ is the convex conjugate of $\Phi$ with respect to its second argument. Since $\Phi(t,0)=0$, we have that $\Phi^*(t,\cdot) \geq 0$ and thus the first integral is well-defined in $\R \cup \{ \infty \}$.

\begin{lemma}
Let $(\bar{u},\bar{m},\bar{v},\bar{P})$ be the solution to \eqref{MFGC}. Let $\tilde{f}$ be defined by $\tilde{f}(x,t)= f(x,t,\bar{m}(t))$. Then, $(\bar{u},\bar{P},\tilde{f})$ is a solution to the following problem:
\begin{equation}
  \label{equ-dual}
\max_{\begin{subarray}{c} u \in W^{2,1,p}(Q) \\ P \in
    L^\infty(0,T;\R^k) \\ \gamma \in L^\infty(Q) \end{subarray}} \! \!
D(u,P,\gamma), \
\text{s.t. }
\begin{cases}
\begin{array}{rl}
-\partial_t u - \sigma \Delta u + H(\nabla u + \phi^\intercal P) \leq
  &
\gamma \\
{u}(x,T) \leq & g(x).
\end{array}
\end{cases}
\end{equation}
Moreover, for all solutions $(u,P,\gamma)$ to the dual problem, $P=\bar{P}$.
If in addition, $\gamma= \tilde{f}$ and the above inequalities hold
as equalities, then $u= \bar{u}$.
\end{lemma}

\begin{proof}
For all $t \in [0,T]$, we have
\begin{equation} \label{eq:dual_1}
- \int_0^T \Phi^*(P) \; \mathrm{d} t
= \int_0^T \Phi(\smint \phi \bar{v} \bar{m}) \; \mathrm{d} t - \iint_Q \langle \phi^{\intercal} P, \bar{v} \bar{m} \rangle \; \mathrm{d} x \; \mathrm{d} t + (a)
\end{equation}
with
\begin{equation*}
(a)= \int_0^T - \Phi^*(P) - \Phi(\smint \phi \bar{v} \bar{m}) + \langle P, \smint \phi \bar{v} \bar{m} \rangle \; \mathrm{d} t
\leq 0.
\end{equation*}
We also have that
\begin{equation} \label{eq:dual_15}
-\int_0^T F^*(t,\gamma(t)) \; \mathrm{d} t + \iint_Q \gamma(x,t) \bar{m}(x,t) \; \mathrm{d} x \; \mathrm{d} t = \int_0^T F(t,\bar{m}(t)) \; \mathrm{d} t + (b),
\end{equation}
where
\begin{equation*}
(b)= \iint_Q \gamma(x,t)\bar{m}(x,t) \; \mathrm{d} x \; \mathrm{d} t - \int_0^T F(t,\bar{m}(t)) \; \mathrm{d} t - \int_0^T F^*(t,\gamma(t)) \; \mathrm{d} t \leq 0.
\end{equation*}
Integrating by parts (in time), we obtain that
\begin{align*}
& \int_{\statespace} {u}(x,0) m_0(x) \; \mathrm{d} x
= \iint_Q - \partial_t u \bar{m} - u \partial_t \bar{m} \; \mathrm{d} x \; \mathrm{d} t
+ \int_{\statespace} u(x,T) \bar{m}(x,T) \; \mathrm{d} x \\
& \qquad  = \ \iint_Q (\sigma \Delta {u} + \gamma
      - H(\nabla {u} + \phi^\intercal {P})) \bar{m}
       + (-\sigma \Delta \bar{m} + \text{div}(\bar{v}\bar{m})) {u} \; \mathrm{d} x \; \mathrm{d} t \\
& \qquad \qquad + \int_{\statespace} g(x) \bar{m}(x,T) \; \mathrm{d} x 
+ (c) + (d),
\end{align*}
where
\begin{align*}
(c) = & \ \iint_Q (-\partial_t {u} - \sigma \Delta {u} +
          H(\nabla {u} + \phi^\intercal {P} ) - \gamma) \bar{m} \; \mathrm{d} x \; \mathrm{d} t \leq 0 \\
(d) = & \ \int_{\statespace} ({u}(x,T)-g(x)) \bar{m}(x,T) \; \mathrm{d} x \leq 0.
\end{align*}
Integrating by parts (in space), we further obtain that
\begin{align}
\int_{\statespace} {u}(x,0) m_0(x) \; \mathrm{d} x
= \ & \iint_Q \big( \gamma - H(\nabla {u} + \phi^\intercal {P})
- \langle \nabla {u}, \bar{v} \rangle \big) \bar{m} \notag \\
& \qquad + \int_{\statespace} g(x) \bar{m}(x,T) \; \mathrm{d} x
+ (c) + (d) \notag \\
= \ & \iint_Q \big( L(\bar{v}) + \gamma \big) \bar{m} + \langle \phi^\intercal \bar{P}, \bar{v} \rangle \bar{m} \; \mathrm{d} x \; \mathrm{d} t \notag \\
& \qquad + \int_{\statespace} g(x) \bar{m}(x,T) \; \mathrm{d} x + (c) + (d) + (e), \label{eq:dual_2}
\end{align}
where
\begin{equation*}
(e)= \iint_Q \big( - H(\nabla {u} + \phi^\intercal {P}) - L(\bar{v}) - \langle \nabla {u} + \phi^\intercal {P}, \bar{v} \rangle \big) \bar{m} \; \mathrm{d} x \; \mathrm{d} t \leq 0.
\end{equation*}
Combining \eqref{eq:dual_1}, \eqref{eq:dual_15} and \eqref{eq:dual_2} together, we finally obtain that
\begin{align*}
D(u,P,\gamma)
= \ & \int_0^T \Phi(\smint \phi \bar{v} \bar{m}) \; \mathrm{d} t + \iint_Q
      L(\bar{v}) \bar{m} \; \mathrm{d} x \; \mathrm{d} t + \int_0^T F(t,m(t)) \; \mathrm{d} t  \\
& \qquad +
\int_{\statespace} g(x) \bar{m}(x,T) \; \mathrm{d} x + (a) + (b) + (c) + (d) + (e) \\
= \ & B(\bar{m},\bar{v}) + (a) + (b) + (c) + (d) + (e).
\end{align*}
The five terms $(a)$, $(b)$, $(c)$, $(d)$, $(e)$ are non-positive and
equal to zero if $(u,P,\gamma)=(\bar{u},\bar{P},\tilde{f})$,
as can be easily verified.
This proves the optimality of $(\bar{u},\bar{P},\tilde{f})$.
Moreover, since $\Phi$ is differentiable (with gradient $\Psi$), the term
$(a)$ is null if and only if
$P(t)= \Psi(\smint \phi \bar{v} \bar{m})= \bar{P}(t)$, for a.e.\@ $t \in [0,T]$.
Therefore, for all optimal solutions $(u,P,\gamma)$, $P= \bar{P}$.
If moreover $\gamma= \tilde{f}$ and the inequality constraints in
\eqref{equ-dual}
hold as equalities,
then (since the HJB equation has a unique solution)
$u= \bar{u}$, which concludes the proof.
\qed \end{proof}

\begin{remark}
It is of interest to check when the density $m(x,t)$ is
a.e.\@ positive, since this is clearly a necessary condition for
the uniqueness of the solution of \eqref{eq:primal_problem_2}.
We note that a sufficient condition for the positivity of $m$
is given in \cite[Proposition 3.10]{POR}.
\end{remark}

\section*{Conclusion}

The existence and uniqueness of a classical solution to a mean field
game of controls have been demonstrated. A particularly important
aspect of the analysis is the fact that the equations $(iii)$ and
$(iv)$\eqref{MFGC}, encoding the coupling of the agents through the
controls, are equivalent to the optimality system of a `static' convex
problem. This observation 
enabled us to eliminate the
variables $v$ and $P$ from the coupled system.

The analysis done in this article can be extended in different ways. A
more complex interaction between the agents could be considered. For
example, it would be possible to replace equations $(iii)$ and $(iv)$ by the following ones:
\begin{equation*}
\begin{array}{rl}
P(t)= & \Psi(t, \int_{\statespace} \varphi(x,t,v(x,t)) m(x,t) \, \mathrm{d} x ) \\[0.5em]
v(x,t)= & -H_p(x,t,\nabla u(x,t) D_v\varphi(x,t,v(x,t)^\intercal P(t)),
\end{array}
\end{equation*}
assuming that $\varphi$ is convex with respect to $v$ and $\Psi \geq 0$.
For a fixed $t \in [0,T]$, this system is equivalent to the optimality system associated with the following convex problem:
\begin{align*}
\inf_{v \colon \statespace \rightarrow \R^d}
& \Phi \Big(t,\int_{\statespace} \varphi(x,t,v(x) ) m(x,t) \, \mathrm{d} x \Big) \\
& \qquad + \int_{\statespace} \big( L(v(x)) + \langle \nabla u(x,t), v(x) \rangle) m(x,t) \; \mathrm{d} x.
\end{align*}
Another possibility of extension of our analysis would be to add convex constraints on the control variable.

Future research will aim at exploiting the potential structure of the problem, which can be used to solve it numerically and to prove the convergence of learning procedures, as was done in \cite{CDHD2015}.

\begin{acknowledgements}
The authors want to thank an anonymous referee
for his useful remarks.
\end{acknowledgements}


\bibliographystyle{abbrv}

    \appendix

\section{A priori bounds for parabolic equations}

In this appendix we provide estimates for the following parabolic equation:
\begin{equation}
  \label{eq:parabolic}
\begin{array}{rll}
\partial_t u - \sigma \Delta u + \langle b, \nabla u \rangle+ c u = & h, \quad & (x,t) \in Q, \\
u(x,0)= & u_0(x), & x \in \statespace,
\end{array}
\end{equation}
for different assumptions on $b$, $c$, $h$, and $u_0$.
The technique is based on the following idea.
By standard parabolic estimates 
detailed below, 
\eqref{eq:parabolic} has a unique solution $u$ in
$L^2(0,T; H^1(\statespace))$,
that we may identify with a periodic function
over $\R^d$.
Let $\varphi :\R^d\rightarrow \R$ be of class $C^\infty$,
with value 1 in a neighbourhood of the closure of
$\statespace$, and with compact support in
$\Omega := B(0,2)$.
Set $Q':=\Omega\times (0,T)$. 
Then $v := u\varphi$ is solution of
\begin{equation}
  \label{eq:parabolic-period}
\begin{array}{rll}
  \partial_t v - \sigma \Delta v + \langle b, \nabla v \rangle + c v
  = & h[u], \quad & (x,t) \in Q', \\
v(x,0)= & v_0(x), & x \in \Omega,
\end{array}
\end{equation}
with $v_0:= u_0\varphi$ and
\be
h[u] := h \varphi - 2 \sigma \langle \nabla \varphi, \nabla u \rangle - \sigma u \Delta \varphi + \langle b, \nabla \phi \rangle u.
\ee
Observe that the solution $v$ of \eqref{eq:parabolic-period}
is equal to 0 in a vicinity of
$(\partial\Omega) \times (0,T)$,
and hence, satisfies the homogeneous Neumann condition;
this allows us to apply some results of \cite{LSU}.

\begin{lemma}
  \label{lem:w21q}
  Let $y \in W^{2,1,q}(Q')$, with $q \in (1,\infty)$. Then
  $y \in L^{q'}(Q')$ and $\nabla y\in L^{q''}(Q')$,
  where
\be
\begin{cases}
\begin{array}{cl}
\frac{1}{q''} = \frac{1}{q} -\frac{1}{d+2}, & \text{if $q < 2+d$,} \\
q'' = \infty, & \text{otherwise},
\end{array}
\end{cases}
\quad
\begin{cases}
\begin{array}{cl}
\frac{1}{q'} = \frac{1}{q} -\frac{2}{d+2}, & \text{if $q< 1 + \frac{d}{2}$,} \\
q'= \infty, & \text{otherwise},
\end{array}
\end{cases}
\ee
with continuous inclusion:
\be
  \label{lem:w21q-2}
  \|y\|_{L^{q'}(Q')} + \| \nabla y \|_{L^{q''}(Q')}
  \leq c(q) \|y\|_{ W^{2,1,q}(Q')}. 
\ee
\end{lemma}

\begin{proof}
See \cite[Lemma 3.3, page 80]{LSU}.
\qed \end{proof}
  
\begin{theorem}
  \label{theo:max_reg-lad}
  Let $q\in (1,\infty)$,
  $w_0\in W^{2-2/q,q}(\Omega)$,
and $h \in L^q(Q')$.
Then the heat equation
  \begin{equation}
  \label{eq:heat}
\begin{array}{rll}
  \partial_t w - \sigma \Delta w 
  = & h, \quad & (x,t) \in Q', \\
w(x,0)= & w_0(x), \quad & x \in \Omega,
\end{array}
\end{equation}
with homogeneous Neumann boundary condition on
$\partial\Omega\times (0,T)$,
  has a unique solution in
  $W^{2,1,q}(Q')$ that satisfies 
\begin{equation*}
  \| w\|_{W^{2,1,q}(Q')} \leq C \big( \| w_0 \|_{W^{2-2/q,q}(\Omega)} + \| h \|_{L^q(Q')} \big).
\end{equation*}
\end{theorem}

\begin{proof}
See \cite[Theorem IV.9.1, page 341]{LSU}.
\qed \end{proof}

\begin{theorem} 
\label{theo:max_reg1}
Let $p > d+2$. For all $R>0$, there exists $C>0$ such that for all $u_0 \in W^{2-2/p,p}(\statespace)$, for all $b \in L^p(Q,\R^d)$, for all $c \in L^p(Q)$, for all $h \in L^p(Q)$, satisfying
\begin{equation*}
\| u_0 \|_{W^{2-2/p,p}(\statespace)} \leq R, \quad
\| b \|_{L^p(Q,\R^d)} \leq R, \quad
\| c \|_{L^p(Q)} \leq R, \quad
\| h \|_{L^p(Q)} \leq R,
\end{equation*}
equation \eqref{eq:parabolic} has a unique solution $u$ in $W^{2,1,p}(Q)$ satisfying moreover $\| u \|_{W^{2,1,p}(Q)} \leq C$.
\end{theorem}

\begin{proof}
  We first check that there is a solution in
  the standard variational setting with 
  spaces $H:=L^2(\statespace)$,
  $V := H^1(\statespace)$. 
  Let us show that, if $y\in V$, then
  $\langle b, \nabla y \rangle$ and $cy$ belong to $V^*$.
  By the Sobolev inclusion,
  $V\subset L^{q_1}(\statespace)$,
  $1/q_1=1/2-1/d$,
  with dense inclusion, so that
  $V^*\subset L^{q_1}(\statespace)^* = L^{q_2}(\statespace)$,
  with $1/q_2=1-1/q_1 = 1/2+1/d$.
  Now $ \langle b, \nabla y \rangle \in L^r(\statespace)$ with
\begin{equation*}
  \frac{1}{r} = \frac{1}{2} + \frac{1}{p} < \frac{1}{2} + \frac{1}{d+2} < \frac{1}{q_2},
  \end{equation*}
  so that
  $\langle b, \nabla y \rangle$ belongs to $V^*$.
  Similarly,
  $c y \in L^r(\statespace)$ with
\begin{equation*}
\frac{1}{r} = \frac{1}{q_1} + \frac{1}{p} < \frac{1}{2} - \frac{1}{d} +\frac{1}{1+d/2}< \frac{1}{q_2},
\end{equation*}
  so that
  $cy$ belongs to $V^*$.  
  So, \eqref{eq:parabolic} has a unique solution in 
  the space
  \be
W(0,T) := \{v \in L^2(0,T;V); \; \partial_t v \in L^2(0,T;V^*) \}.
\ee
Then we easily check that
$h[u] \in L^{q_0}(Q')$,
for some $q_0\in (1,2)$.
Then, by Theorem \ref{theo:max_reg-lad},
$v \in W^{2,1,q_0}(Q)$.  
We next compute by induction a finite sequence $(q_k)_{k=0,1,...,K}$ such that
\begin{equation*}
(i) \ v \in W^{2,1,q_k}(Q'), \ \forall k=0,...,K, \quad
(ii) \ q_k \in (1,d+2), \ \forall k=0,...,K-1, \quad (iii) \ q_K \geq d+2.
\end{equation*}
The first element $q_0$ has already been fixed and satifies $v \in W^{2,1,q_0}(Q')$. If $q_0 \geq d+2$, we can stop and set $K=0$.
Let $k \in \mathbb{N}$, assume that $q_k \in (1,d+2)$ and that $v \in W^{2,1,q_k}(Q')$.
Then $v$
is solution of
\begin{equation}
\label{eq:parabolic-period2}
\begin{array}{rll}
\partial_t v - \sigma \Delta v = & h''[u], \quad & (x,t) \in Q', \\
u(x,0)= & v_0(x), & x \in \Omega,
\end{array}
\end{equation}
where
\be
h''[u] := h \varphi - 2 \sigma \langle \nabla \varphi, \nabla u \rangle - \sigma u \Delta \varphi
+ u \langle b, \nabla \varphi \rangle
- \varphi ( \langle b, \nabla u \rangle + cu).
\ee
We construct now $q_{k+1}$ in such a way that $h''[u] \in L^{q_{k+1}}(Q')$.
Since $v \in W^{2,1,q_k}(Q')$, we have that $u \in W^{2,1,q_k}(Q)$ and thus
by Lemma \ref{lem:w21q}, $\langle b, \nabla u\rangle\in L^{r'}(Q')$ with
\begin{equation} \label{r_prime}
\frac{1}{r'} = \frac{1}{q_k} + \frac{1}{p} - \frac{1}{d+2}.
\end{equation}
If $q_k < 1 + d/2$, then $c u \in L^{r''}(Q')$ with
\begin{equation}
\frac{1}{r''} = \frac{1}{q_k} + \frac{1}{p} - \frac{2}{d+2}.
\end{equation}
Note that $r'' > r'$.
If $q_k \geq 1 + d/2$, then $u \in L^\infty(Q')$ and thus $cu \in L^p(Q')$.
We set now $q_{k+1}= \min(r',p)$. We observe that in both cases, $cu \in L^{q_{k+1}}(Q')$.
One can verify that the other terms of $h''[u]$ also lie in $L^{q_{k+1}}(Q')$.
Therefore, by Theorem \ref{theo:max_reg-lad}, $v\in W^{2,1,q_{k+1}}(Q')$.
If $q_{k+1} \geq d+2$, we stop the construction of the sequence and set $K=k+1$.
It remains to prove that the construction of the sequence stops after finitely many iterations. If that was not the case, we would have that $q_{k+1}= r'$, with $r'$ defined in \eqref{r_prime}, for all $k \in \mathbb{N}$, implying that
\begin{equation*}
\frac{1}{q_k} = \frac{1}{q_0} + k \Big( \frac{1}{p} - \frac{1}{d+2} \Big) \underset{k \to \infty}{\longrightarrow} -\infty,
\end{equation*}
which is a contradiction.
Now we know that $v \in W^{2,1,q_K}(Q')$, with $q_K \geq d+2$. This implies that $u \in L^\infty(Q')$ and $\nabla u \in L^\infty(Q',\R^d)$ (by Lemma \ref{lem:w21q}) and thus that $h''[u] \in L^p(Q')$. Finally, $v \in W^{2,1,p}(Q')$ (by Theorem \ref{theo:max_reg-lad}) and $u \in W^{2,1,p}(Q)$, since $u$ and $v$ coincide on $Q$.

Observing that $q_0$,...,$q_K$ only depend on $p$ and $d$, the reader can check that $v$ (and thus $u$) can be bounded in $W^{2,1,p}(Q')$ by a constant depending on $R$ only.
\qed \end{proof}

\begin{theorem}
\label{LAD-l3.4p82}
For $q \in (1,\infty)$, the trace at time $t=0$ of elements of $W^{2,1,q}(Q')$
belongs to $W^{2-2/q,q}(\Omega)$. 
\end{theorem}

\begin{proof}
See \cite[Lemma 3.4, page 82]{LSU}.
\qed \end{proof}

\begin{theorem} \label{theo:max_reg2}
Let $p > d+2$. There exists $C>0$ such that for all $u_0 \in W^{2-2/p,p}(\statespace)$ and for all $h \in L^p(Q)$, the unique solution $u$ to \eqref{eq:parabolic} (with $b= 0$ and $c=0$) satisfies the following estimate:
\begin{equation*}
\| u \|_{W^{2,1,p}(Q)} \leq C \big( \| u_0 \|_{W^{2-2/p,p}(\statespace)} + \| h \|_{L^p(Q)} \big).
\end{equation*}
\end{theorem}

\begin{proof}
Consider the map
$u \in W^{2,1,p}(Q) \mapsto (u(\cdot,0), \partial_t u - \sigma \Delta u - h)
\in W^{2-2/p,p}(\Omega),L^p(Q)).$
By Theorem \ref{LAD-l3.4p82}, it is continuous and by Theorem \ref{theo:max_reg2}, it is bijective.
As a consequence of the open mapping theorem,
its inverse is also continuous.
The result follows.
\qed \end{proof}

\begin{lemma}
  \label{lemma:max_reg_embedding}
Let $p>d+2$.
There exists $\delta \in (0,1)$ and $C>0$ such that for all $u \in
W^{2,1,p}(Q)$,
\begin{equation*}
\| u \|_{\mathcal{C}^\delta(Q)} + \| \nabla u \|_{\mathcal{C}^\delta(Q,\R^d)}
\leq C \| u \|_{W^{2,1,p}(Q)}.
\end{equation*}
\end{lemma}

\begin{proof}
See \cite[Lemma II.3.3, page 80 and Corollary, page 342]{LSU}.
\qed \end{proof}

  \begin{theorem}
    \label{theo:holder_reg_classical}
    Let $p>d+2$.
    For all $\alpha \in (0,1)$, for all $R>0$, there exist $\beta \in (0,1)$ and $C>0$ such that for all $u_0 \in \mathcal{C}^{2+ \alpha}(\statespace)$, $b \in \mathcal{C}^{\alpha,\alpha/2}(Q,\R^d)$, $c \in \mathcal{C}^{\alpha,\alpha/2}(Q)$ and $h \in \mathcal{C}^{\alpha,\alpha/2}(Q)$ satisfying
\begin{equation*}
\| u_0 \|_{\mathcal{C}^{2+ \alpha}(\statespace)} \leq R, \ \
\| b \|_{\mathcal{C}^{\alpha,\alpha/2}(Q,\R^d)} \leq R, \ \
\| c \|_{\mathcal{C}^{\alpha,\alpha/2}(Q)} \leq R, \ \ \text{and} \ \
\| h \|_{\mathcal{C}^{\alpha,\alpha/2}(Q)} \leq R,
\end{equation*}
the solution to \eqref{eq:parabolic} lies in $\mathcal{C}^{2+\beta,1+\beta/2}(Q)$ and satisfies
$\| u \|_{\mathcal{C}^{2+\beta,1+\beta/2}(Q)} \leq C$.
\end{theorem}

\begin{proof}
In the proof, $C$ denotes constants
    that depend only on $\alpha$ and $R$.
Combining Theorem \ref{theo:max_reg1}
and Lemma \ref{lemma:max_reg_embedding},
  we obtain that $h[u]$ is H\"older continuous,  
    with exponent $\beta:=\min(\delta,\alpha)$
    (where $\delta$ is given by 
    Lemma \ref{lemma:max_reg_embedding};
    we use the fact that a product of H\"older functions is H\"older,
    with exponent equal to the minimum exponent),
    and $\| h[u]\|_{\mathcal{C}^{\beta,\beta/2}(Q)} \leq C.$
    By \cite[Theorem IV.5.1, page 320]{LSU},
        $\| v \|_{\mathcal{C}^{2+\beta,1+\beta/2}(Q)} \leq C$.
  Since $u$ and $v$ coincide on
  $\statespace$, the conclusion follows.
\qed \end{proof}

 \end{document}